\documentclass[final]{siamart1116}

\usepackage{bm,amsfonts,amssymb,amsmath}
\usepackage{mathrsfs,algorithmic}
\usepackage{graphicx,color,version}
\usepackage{indentfirst,enumerate}
\usepackage{multirow}
\newtheorem{rem}{Remark}[section]

\newcommand{\md}{\mathrm{d}}
\newcommand\opG{{\mathcal{G}}}
\newcommand\En{{\mathcal{E}}}
\newcommand\opL{{\mathcal{L}}}

\numberwithin{theorem}{section}
\numberwithin{equation}{section}
\newtheorem{example}{\textit{Example}}

\definecolor{orange}{rgb}{1,0.5,0}
\definecolor{rb}{rgb}{1,0,1}
\allowdisplaybreaks[4]

\begin{document}

% Sets running headers as well as PDF title and authors
\headers{Energy stable scheme for gradient flows}{J. Shen, J. Xu, J. Yang}

% Title. If the supplement option is on, then "Supplementary Material"
% is automatically inserted before the title.
\title{{A new class of efficient and robust energy stable schemes for gradient flows}\thanks{This work is partially supported  by DMS-1620262, DMS-1720442 and AFOSR FA9550-16-1-0102.}}

% Authors: full names plus addresses.
\author{
  Jie Shen\thanks{Department of Mathematics, Purdue University, West Lafayette, IN 47907, USA (\email{shen7@purdue.edu}).}
  \and
  Jie Xu\thanks{Department of Mathematics, Purdue University, West Lafayette, IN 47907, USA (\email{xu924@purdue.edu}).}
  \and
  Jiang Yang\thanks{Department of Mathematics, Southern University of Science and Technology, Shenzhen, Guangdong 518000, China (\email{yangj7@sustc.edu.cn}).}
}

\maketitle

\begin{abstract}
We propose a new  numerical technique to  deal with nonlinear terms in  gradient flows. By introducing a scalar auxiliary
variable (SAV), we construct efficient and robust energy
stable schemes for a large class of gradient flows. The SAV approach is not restricted to specific forms of the
nonlinear part of the free energy, and only requires to solve {\it decoupled}
linear equations with {\it constant coefficients}. 
We use this technique to deal with  several challenging applications
which can not be easily handled by existing approaches, and present
convincing numerical results to show that our schemes are not only much
more efficient and easy to implement, but   can also better capture
the physical properties in these models. Based on this SAV approach, we can construct unconditionally second-order energy stable schemes; and 
we can easily construct even third or fourth order BDF schemes, although not unconditionally stable, 
which are very robust in practice. In particular, when coupled with an  adaptive time stepping strategy, the SAV approach can be extremely efficient and accurate. 

\end{abstract}

\begin{keywords}
  gradient flows;  energy stability;  Allen--Cahn and Cahn--Hilliard equations; phase field models; nonlocal models. 
\end{keywords}

\begin{AMS}
65M12;  35K20; 35K35; 35K55; 65Z05.
\end{AMS}

\section{Introduction}
Gradient flows are dynamics driven by a free energy. 
Many physical problems can be modeled  by PDEs that take the form of gradient flows, which are often derived from the second law of thermodynamics. 
Examples of these problems include interface dynamics \cite{anderson1998diffuse,gurtin1996two,kim2012phase,liu2003phase,lowengrub1998quasi,yue2004diffuse}, crystallization \cite{elder2002modeling,elder2004modeling,elder2007phase}, thin films \cite{giacomelli2001variatonal,otto1998lubrication}, polymers \cite{maurits1997mesoscopic,fraaije1993dynamic,fraaije2003model,fraaije1997dynamic} and liquid crystals \cite{leslie1979theory,doi1988theory,larson1990arrested,larson1991effect,forest2004weak,forest2004flow,qian1998generalized,yu2010nonhomogeneous}. 

A gradient flow is determined not only by the driving free energy, but also the dissipation mechanism. 
Given a free energy functional $\En [\phi(\bm{x})]$ bounded from below. 
Denote its variational derivative as $\mu=\delta \En /\delta \phi$. 
The general form of the gradient flow can be written as
\begin{equation}\label{Gflow}
  \frac{\partial \phi}{\partial t}=\opG\mu, 
\end{equation}
supplemented with suitable boundary conditions. To simplify the presentation, we   assume  throughout the paper that the boundary conditions are chosen such that all boundary terms will vanish when integrating by parts are performed. This is true with periodic boundary conditions or homogeneous Neumann boundary conditions.  

In the above, a non-positive symmetric operator $\opG$ gives the dissipation mechanism. 
The commonly adopted dissipation mechanisms include the $L^2$ gradient flow where $\opG=-I$, the $H^{-1}$ gradient flow where $\opG=\Delta$, or more generally  non-local $H^{-\alpha}$ gradient flow where $\opG=-(-\Delta)^\alpha$ $(0< \alpha <1)$ (cf. \cite{Ain.M17}). 
For more complicated dissipation mechanisms, $\opG$ may be nonlinear and may depend on $\phi$. An example is the Wasserstein gradient flow for $\phi>0$, where $\opG \mu=\nabla\cdot(\phi\nabla \mu)$ (cf. \cite{doi1988theory,jordan1998variational}). 
As long as $\opG$ is non-positive, the free energy is non-increasing, 
\begin{equation}
  \frac{\md \En[\phi]}{\md t}=\frac{\delta\En}{\delta\phi}\cdot\frac{\partial \phi}{\partial t}=(\mu,\opG\mu)\le 0, \label{energy_cont}
\end{equation}
where $(\phi,\psi)=\int_\Omega \phi \psi \md\bm{x}$.
In this paper, we will focus on the case where $\opG$ is non-positive, linear and independent of $\phi$. 

Although gradient flows take various forms, from the numerical perspective, a scheme is generally evaluated from the following  aspects: 
\begin{enumerate}[(i)]
\item whether the scheme keeps the energy dissipation; 
\item the order of its accuracy; 
\item its efficiency; 
\item whether the scheme is easy to implement. 
\end{enumerate}
Among these the first aspect is particularly important, and is crucial to eliminating  numerical results that are not physical. 
Oftentimes, if this is not put into thorough consideration when constructing the scheme, it may require a time step extremely small to keep the energy dissipation. 

Usually, the free energy functional contains a quadratic term, 
which we write explicitly as 
\begin{equation}
  \En[\phi]=\frac{1}{2}(\phi,\opL\phi)+\En_1[\phi], \label{energy0}
\end{equation}
where $\opL$ is a symmetric non-negative linear operator (also independent of $\phi$), and $\En_1[\phi]$   are nonlinear but with only lower-order derivatives than $\opL$ . 
To obtain an energy dissipative scheme, the linear term is usually treated implicitly in some manners, while different approaches have to be used 
 for  nonlinear terms. 
In the next few paragraphs, we briefly review the existing approaches for dealing with the nonlinear terms. 

The first approach is the convex splitting method which was perhaps first introduced in \cite{MR1249036} but popularized by \cite{eyre1998unconditionally}. 
If we can express the free energy as the difference of two convex functional, namely $\En=\En_c-\En_e$ where both $\En_c$ and $\En_e$ are convex about $\phi$, then a simple convex splitting scheme reads
\begin{equation}
  \frac{\phi^{n+1}-\phi^n}{\Delta t}
  =\opG\left(\frac{\delta \En_c}{\delta\phi}[\phi^{n+1}]
  -\frac{\delta \En_e}{\delta\phi}[\phi^{n}]\right). \label{CnvSpl}
\end{equation}
By using the property of convex functional, 
$$
\En_c[\phi_2]-\En_c[\phi_1]\ge \frac{\delta\En_c}{\delta \phi}[\phi_1](\phi_2-\phi_1), 
$$
and  multiplying \eqref{CnvSpl} with $({\delta \En_c}/{\delta\phi})[\phi^{n+1}]-({\delta \En_e}/{\delta\phi})[\phi^{n}]$,
it is easy to check that the scheme satisfies the discrete energy law $\En[\phi^{n+1}]\le \En[\phi^n]$ unconditionally. 
Because the implicit part $\delta F_c/\delta\phi$ is usually nonlinear about $\phi$, 
we need to solve nonlinear equations at each time step, which can be expensive. The scheme \eqref{CnvSpl} is only first-order. While it is possible to construct second-order convex splitting schemes for certain situations on a case by case basis  (see, for instance, \cite{shen2012second,MR3118257,MR3164683}),
 a general formulation of second-order convex splitting schemes is not available. 

The second approach is the so-called stabilization method which  treats the nonlinear terms explicitly, and  add a stabilization term to avoid strict time step constraint \cite{ZCST99, shen2010numerical}. More precisely,
if we can find a simple linear operator $\tilde{\opL}$ such that both $\tilde{\opL}$ and $\tilde{\opL}-(\delta^2\En_1/\delta\phi^2)[\phi]$ are positive, 
then we may choose a particular convex splitting, 
$$
\En_c=\frac{1}{2}(\phi,\opL\phi)+\frac{1}{2}(\phi,\tilde{\opL}\phi),\quad 
\En_e=\frac{1}{2}(\phi,\tilde{\opL}\phi)+\En_1[\phi], 
$$
which leads to the following unconditionally energy stable scheme:
\begin{equation}
  \frac{\phi^{n+1}-\phi^n}{\Delta t}
  =\opG\left(\opL \phi^{n+1}+\frac{\delta \En_1}{\delta \phi}[\phi^n]+\tilde{\opL}(\phi^{n+1}-\phi^n)\right). 
\end{equation}
Hence, the stabilization method is in fact a special class of convex splitting method.
A common choice of $\tilde{\opL}$ is 
$$
\tilde{\opL}=a_0+a_1(-\Delta)+a_2(-\Delta)^2+\ldots. 
$$
The advantage of the stabilization method is that when the dissipation operator $\opG$ is also linear, we only need to solve a linear system like $(1-\Delta t\opG(\opL+\tilde{\opL}))\phi^{n+1}=b^n$ at each time step. 
However, it is not always the case that $\tilde{\opL}$ can be found. 
The stabilization method can be extended to second-order schemes, 
but in general it can not be unconditionally energy stable, see however a recent work in \cite{Li.Q17}.
% One can  use a spectral deferred correction methods based on the first-order stabilized schemes to achieve high-order accuracy \cite{Liu.S15,yang2015SDC}.
 On the other hand, a related method is the exponential time differencing (ETD) approach in which the operator $\tilde{\opL}$ is integrated exactly (see, for instance, \cite{MR3299200} for an example on related applications).

The third approach is the method of invariant energy quadratization (IEQ), which  was proposed very recently in \cite{yang2016linear,zhao2016numerical}. This method is a generalization of the method of Lagrange multipliers or of auxiliary variables originally proposed in \cite{BGG11,Gui.T13}.
In this approach, $\En_1$ is assumed to take the form $\En_1[\phi]=\int_\Omega g(\phi) \md\bm{x}$ where $g(\phi)\ge -C_0$ for some $C_0>0$. 
One then introduces an auxiliary variable $q=\sqrt{g+C_0}$, and transform \eqref{Gflow} into a equivalent system, 
\begin{subequations}\label{IEQeq}
\begin{align}
  \frac{\partial\phi}{\partial t}=&\opG\left(\opL\phi+\frac{q}{\sqrt{g(\phi)+C_0}}g'(\phi)\right),\label{IEQeq_phi} \\
  \frac{\partial q}{\partial t}=&\frac{g'(\phi)}{2\sqrt{g(\phi)+C_0}}\frac{\partial\phi}{\partial t}. \label{IEQeq_r}
\end{align}
\end{subequations}
Using the fact that $\En_1[\phi]=\int_\Omega q^2 \md\bm{x}$ is convex about $q$, 
we can easily construct simple and linear energy stable schemes.  For instance, 
a first-order scheme is given by 
\begin{subequations}\label{IEQ}
\begin{align}
  \frac{\phi^{n+1}-\phi^{n}}{\Delta t}=&\opG\mu^{n+1},\label{IEQ1}\\
  \mu^{n+1}=&\opL\phi^{n+1}+\frac{q^{n+1}}{\sqrt{g(\phi^n)+C_0}}g'(\phi^n),\label{IEQ2}\\
  \frac{q^{n+1}-q^n}{\Delta t}=&\frac{g'(\phi^n)}{2\sqrt{g(\phi^n)+C_0}}\frac{\phi^{n+1}-\phi^{n}}{\Delta t}.\label{IEQ3}
\end{align}
\end{subequations}
One can easily show that the above scheme is unconditionally energy stable.
% Indeed,  taking the inner product of the above with $\mu^{n+1}$, $\frac{\phi^{n+1}-\phi^{n}}{\Delta t}$ and $2q^{n+1}$, we obtain immediately the following energy law:
%\begin{align}
 % \frac{1}{2}(\phi^{n+1},\opL\phi^{n+1})&+||q^{n+1}||^2-\frac{1}{2}(\phi^n,\opL\phi^{n})-||q^{n}||^2\nonumber\\
%  &+\frac{1}{2}\big(\phi^{n+1}-\phi^n,\opL(\phi^{n+1}-\phi^n)\big)
%  +||q^{n+1}-q^n||^2\le \Delta t(\mu^{n+1}, \opG\mu^{n+1}). 
%\end{align}
%Thus, modified free energy $\frac{1}{2}(\phi^n,\opL\phi^{n})+||q^{n}||^2$ is decreasing with $n$.
Furthermore, eliminating $q^{n+1}$ and $\mu^{n+1}$,  we obtain a linear system for $\phi^{n+1}$ in the following form:
\begin{equation}
  \left(\frac{1}{\Delta t}-\opG\opL-\opG\frac{(g'(\phi^n))^2}{2g(\phi^n)}\right)\phi^{n+1}=b^n. \label{IEQ_linear}
\end{equation}
Similarly, one can also construct unconditionally energy stable second-order schemes. 
The IEQ approach is remarkable as it allows us to construct linear, unconditionally stable, and second-order unconditionally energy stable schemes for a large class of gradient flows. However, it still suffer from the following drawbacks:
\begin{itemize}
\item Although one only needs to solve a linear system at each time step, the linear system usually involves  variable coefficients  which change at each time step. 
%Even if the two operators $\opG$ and $\opL$ are commutative, for example $\opG=\opL=-\Delta$, the term $\opG\frac{(g'(\phi^n))^2}{2g(\phi^n)}$ may still make the matrix asymmetric. 
\item For gradient flows with multiple components, the IEQ approach will lead to coupled systems with variable coefficients. 
\item It requires that the energy density function $g(\phi)$ is bounded from below, while in some case, one can only assume that $\En_1[\phi]=\int_\Omega g(\phi) \md\bm{x}$ is bounded from below.
\end{itemize}
In \cite{SXY17}, we introduced the so-called 
 scalar auxiliary variable (SAV) approach, which inherits all advantages of IEQ approach but also overcome most of its shortcomings. 
More precisely, by using the Cahn-Hilliard equation and a system of Cahn-Hilliard equations as examples, we showed that the SAV approach has the following advantages: 
\begin{enumerate}[(i)]
\item 
For single-component gradient flows, it leads to, at each time step,  linear equations with constant coefficients so it is remarkably easy to implement. 
\item For multi-component gradient flows, it leads to, at each time step, decoupled linear equations with constant coefficients, one for each component. 
\end{enumerate}
\begin{comment}
At each time step we only need to solve a linear system of the form 
$(1-c_0\Delta t\opG\opL)\bar x=\bar b$ twice, where $c_0$ is positive constant depending on the time discretization scheme. 
More precisely, the task for solving the gradient flow using the new SAV approach is equivalent to using an implicit scheme to solve the linear parabolic PDE:
$
\frac{\partial \phi}{\partial t}=\opG\opL \phi.
$ 
\iffalse
$$
\frac{\phi^{n+1}-\phi^n}{\Delta t}=\opG\opL \phi^{n+1}. 
$$
\fi
If $\opG$ and $\opL$ are commutative, the linear system is symmetric; even if they are not commutative, the matrix does not vary with time. 

\item 
The SAV approach only requires the free energy $\En_1[\phi]=\int_\Omega g(\phi)\md \bm{x}$ to be bounded below, instead of a uniform lower bound for the free energy density function $g(\phi)$,   enabling us to deal with a larger class of free energies. 
\end{enumerate}
\end{comment}
The main goals of this paper are (i) to expand the SAV approach to a more general setting, and apply it to several challenging applications, such as non-local phase field crystals, Molecular beam epitaxial without slope section, a Q-tensor model for liquid crystals,  a phase-field model for two-phase incompressible flows;
% which are difficult to deal with using other existing approaches;
(ii) to  show that, besides its simplicity and efficiency, the novel schemes  present better accuracy  compared with other schemes; and
 (iii) to validate the effectiveness and robustness of the SAV approach coupled with high-order BDF schemes and adaptive time stepping.

We emphasize that the schemes  are formulated in a general form that are  applicable to a large class of gradient flows. 
We also suggest some criteria on the choice of $\opL$ and $\En_1$, which is useful when attempting to construct numerical schemes for particular gradient flows. %Since the advantage of IEQ approach over other approaches such as the convex splitting and stabilization has already been demonstrated for many situations \cite{??}, we shall only compare the SAV approach with the IEQ approach.

The rest of paper is organized as follows. In Section \ref{Schm}, we describe the construction of SAV schemes for gradient flows in a general form. In Section 3, we present several numerical examples to validate the SAV approach.
In Section 4, we describe how to construct higher-order SAV schemes and how to implement adaptive time stepping. 
We then apply the SAV approach to construct second-order unconditionally stable, decoupled linear schemes for several challenging situations in Section 5, followed by some
concluding remarks in Section 6. 

\section{SAV approach for constructing energy stable schemes\label{Schm}}
In this section, we formulate the SAV approach introduced in \cite{SXY17} for a class of general gradient flows.

\subsection{Gradient flows of a single function}
We consdier \eqref{Gflow} with free energy in the form of  \eqref{energy0} such that
 $\En_1[\phi]$ is bounded from below. 
Without loss of generality, we assume that $\En_1[\phi]\ge C_0>0$, otherwise we may add a constant to $\En_1$ without altering the gradient flow. 
We introduce a scalar auxiliary  variable $r=\sqrt{\En_1}$,  
and rewrite the gradient flow \eqref{Gflow} as
\begin{subequations}\label{SAV1}
\begin{align}
 & \frac{\partial\phi}{\partial t}=\opG\mu,\label{SAVeq_mu}\\
& \mu=\opL\phi+\frac{r}{\sqrt{\En_1[\phi]}}U[\phi], \label{SAVeq_phi}\\
 & r_t=\frac{1}{2\sqrt{\En_1[\phi]}}\int_\Omega  U[\phi]\phi_t \md\bm{x}, \label{SAVeq_r}
\end{align}
\end{subequations}
where
\begin{equation}
  U[\phi]=\frac{\delta \En_1}{\delta \phi}. 
\end{equation}
Taking the inner products of the above with $\mu$, $\frac{\partial\phi}{\partial t}$ and  $2r$, respectively,  we obtain the energy dissipation law for \eqref{SAV1}:
\begin{equation}
 \frac{d}{dt}[(\phi, \opL\phi)+ r^2]=(\mu,\opG\mu).
\end{equation}
Note that this equivalent system  \eqref{SAV1} is similar to the system \eqref{IEQeq_phi}  and \eqref{IEQeq_r} in the IEQ approach, except that a scalar auxiliary  variable  $r$ is introduced  instead of a function $q(\phi)$. 
To illustrate the advantage of  SAV  over IEQ, we start from a first-order scheme: 
\begin{subequations}\label{SAV_first}
\begin{align}
  \frac{\phi^{n+1}-\phi^{n}}{\Delta t}=&\opG\mu^{n+1}, \label{SAV_first1}\\
  \mu^{n+1}=&\opL\phi^{n+1}+\frac{r^{n+1}}{\sqrt{\En_1[\phi^n]}}U[\phi^n],\label{SAV_first2}\\
  \frac{r^{n+1}-r^n}{\Delta t}=&\frac{1}{2\sqrt{\En_1[\phi^n]}}\int_\Omega U[\phi^n]\frac{\phi^{n+1}-\phi^{n}}{\Delta t}\md\bm{x}.\label{SAV_first3}
\end{align}
\end{subequations}
Multiplying the three equations with $\mu^{n+1}$, $(\phi^{n+1}-\phi^n)/\Delta t$, $2r^{n+1}$, integrating the first two equations, and adding them together, we obtain the discrete energy law: 
\begin{align*}
  \frac{1}{\Delta t}&\Big[\tilde\En(\phi^{n+1},r^{n+1})-\tilde\En(\phi^n,r^n)\big]\\
  &+ \frac{1}{\Delta t}\Big[\frac{1}{2}(\phi^{n+1}-\phi^n,\opL(\phi^{n+1}-\phi^n))+(r^{n+1}-r^n)^2\Big]=(\mu^{n+1},\opG\mu^{n+1}),\label{energy_discrete_first}
\end{align*}
where we defined  a modified energy 
\begin{equation}\label{modE}
 \tilde \En(\phi, r)=\frac12(\phi,\opL\phi)+r^2.
\end{equation}
Thus, the scheme is unconditionally energy stable with the modified energy. Note that, while $r=\sqrt{\En(\phi)}$,  we do not  have $r^n=\sqrt{\En(\phi^n)}$ so the modified energy $\tilde \En(\phi^n, r^n)$ is different  from the original energy $\En(\phi^n)$.

\begin{rem}
 Notice that the SAV scheme \eqref{SAV_first} is unconditionally energy stable (with a  modified energy) for arbitrary energy splitting in \eqref{energy0} as long as $\En_1$ is bounded from below. One might wonder why not taking $\opL=0$? Then, the scheme \eqref{SAV_first} would be totally explicit, i.e., without having to solve any equation, but unconditionally energy stable (with a  modified energy $r^2(t)$)!
However, energy stability alone is not sufficient for convergence.
Such scheme will not be able to produce meaningful results, since the modified energy \eqref{modE} reduces to $r^2(t)$ which can not control any oscillation 
due to derivative terms. Hence, it is necessary  that  $\opL$ contains enough  dissipative terms (with at least linearized highest derivative terms).
\end{rem}
 
An important fact is that the SAV scheme \eqref{SAV_first} is easy to implement. 
\begin{comment}
Indeed,
we can rewrite it as a matrix system
\begin{equation}
   \begin{matrix} % or pmatrix or bmatrix or Bmatrix or ...
      I  & -\Delta t \opG & 0 \\
      -\opL  & I &\frac{U[\phi^n]}{\sqrt{\En_1[\phi^n]}} \\
     - \frac{U[\phi^n]}{\sqrt{\En_1[\phi^n]}}  & 0 & I\\
   \end{matrix}   
    \begin{matrix} % or pmatrix or bmatrix or Bmatrix or ...
      \phi^{n+1} \\
      \mu^{n+1}\\
      r^{n+1}
   \end{matrix} 
   =\bar b^n,
    \end{equation}
    where $\bar b^n$ contains terms at the time step $n$. Hence, we can perform a block-Gaussian elimination to solve $r^{n+1}$ first, and then obtain  $(\phi^{n+1},\mu^{n+1})$ by 
    \end{comment}
    Indeed, taking \eqref{SAV_first2} and \eqref{SAV_first3} into \eqref{SAV_first1}, we obtain 
\begin{align}
  \frac{\phi^{n+1}-\phi^n}{\Delta t}=&\opG\left[\opL \phi^{n+1}%\nonumber\\
  +\frac{U[{\phi}^{n}]}{\sqrt{\En_1[{\phi}^{n}]}}
  \left(r^n+\int_\Omega~\frac{U[{\phi}^{n}]}{2\sqrt{\En_1[{\phi}^{n}]}}(\phi^{n+1}-\phi^n)\md\bm{x}\right)\right]. 
\end{align}
Denote 
\begin{align}
b^n=U[{\phi}^{n}]/\sqrt{\En_1[{\phi}^{n}]}. \nonumber
\end{align}
Then the above equation can be written as 
\begin{equation}
  (I-\Delta t \opG\opL)\phi^{n+1}-\frac{\Delta t}{2}\opG b^n(b^n,\phi^{n+1})=\phi^{n}+r^n\opG b^n-\frac{\Delta t}{2}(b^n,\phi^n)\opG b^n. \label{linear1}
\end{equation}
The above equation  can be solved using the Sherman--Morrison--Woodbury formula \cite{Gol.V89}:
\begin{equation}\label{SMW}
 (A+{U}{V}^T)^{-1}=A^{-1}-A^{-1}{U}(I+{V}^TA^{-1}{U})^{-1}{V}^TA^{-1},
\end{equation}
where $A$ is an $n\times n$ matrix,  $U$ and $V$ are $n\times k$ matrices, and $I$ is the $k\times k$ identity matrix. We note that if $k\ll n$ and $A$ can be inverted efficiently, the Sherman--Morrison--Woodbury formula provides an efficient algorithm to invert the perturbed matrix $A+{U}{V}^T$.
The system \eqref{linear1} corresponds to a case with $U$ and $V$ being $n\times 1$  vectors, so it can be efficiently solved by using \eqref{SMW}. For the reader's convenience, we write down explicitly the procedure below. 
Denote the righthand side of \eqref{linear1} by $c^n$.
Multiplying \eqref{linear1} with $(I-\Delta t \opG\opL)^{-1}$, then taking the inner product with $b^n$, we obtain 
\begin{equation}
  (b^n,\phi^{n+1})+\frac{\Delta t}{2}\gamma^n(b^n,\phi^{n+1})=(b^n,(I-\Delta t \opG\opL)^{-1}c^n), 
\end{equation}
where $\gamma^n=-(b^n,(I-\Delta t \opG\opL)^{-1}\opG b^n)=(b^n,(-\opG^{-1}+\Delta t \opL)^{-1} b^n)>0$, if we assume that $\opG$ is negative definite and $\opL$ is non-negative. 
Hence
\begin{equation}\label{bnphi}
  (b^n,\phi^{n+1})=\frac{(b^n,(I-\Delta t \opG\opL)^{-1}c^n)}{1+\Delta t\gamma^n/2}. 
\end{equation}
%Substituting it into \eqref{linear1}, we deduce that 
%\begin{equation}
%  \phi^{n+1}=(I-\Delta t \opG\opL)^{-1}c^n-\frac{(b^n,(I-\Delta t \opG\opL)^{-1}c^n)}{1+\Delta t\gamma^n/2}\cdot\frac{\Delta t}{2}(I-\Delta t \opG\opL)^{-1}\opG b^n. 
%\end{equation}
To summarize, we implement \eqref{SAV_first} as follows:
\begin{enumerate}[(i)]
 \item Compute $b^n$ and $c^n$ (the righthand side of  \eqref{linear1});
 \item Compute $(b^n,\phi^{n+1})$ from \eqref{bnphi};
 \item Compute $\phi^{n+1}$ from   \eqref{linear1}.
\end{enumerate}
Note that in (ii) and (iii) of the above procedure, we only need to solve, twice, a linear equation with constant coefficients of the form 
\begin{equation}\label{common}
  (I-\Delta t \opG\opL)\bar x=\bar b. 
\end{equation}
Therefore, the above procedure is extremely efficient.
In particular,  if $\opL=-\Delta$ and $\opG=-1$ or $-\Delta$, with a tensor-product domain $\Omega$,  fast solvers are available. 
In contrast, the convex splitting schemes usually require solving a nonlinear system,  the IEQ scheme requires solving \eqref{IEQ_linear} which involves variable coefficients. 

%\subsection{Second-order scheme}
A main advantage of the SAV approach (as well as the IEQ approach) is that linear second- or even higher-order energy stable schemes can be easily constructed.
We start by a semi-implicit second-order scheme based on  
 Crank--Nicolson, which we denote as SAV/CN:
 \begin{subequations}\label{second}
\begin{align}
  \frac{\phi^{n+1}-\phi^n}{\Delta t}=&\opG\mu^{n+1/2}, \label{second_1}\\
  \mu^{n+1/2}=&\opL\frac{1}{2}(\phi^{n+1}+\phi^n)
  +\frac{r^{n+1}+r^n}{2\sqrt{\En_1[\bar{\phi}^{n+1/2}]}}
  U[\bar{\phi}^{n+1/2}], \label{second_2}\\
  r^{n+1}-r^n=&\int_\Omega\frac{U[\bar{\phi}^{n+1/2}]}{2\sqrt{\En_1[\bar{\phi}^{n+1/2}]}}
  (\phi^{n+1}-\phi^n)\md\bm{x}. \label{second_3}
\end{align}
\end{subequations}
In the above, $\bar{\phi}^{n+1/2}$ can be any explicit approximation of $\phi(t^{n+1/2})$ with an error of $O(\Delta t^2)$. For instance, we may let 
\begin{equation}
  \bar{\phi}^{n+1/2}=\frac{1}{2}(3\phi^{n}-\phi^{n-1}) 
\end{equation}
be the extrapolation; 
or we can use a simple first-order scheme to obtain it, such as the semi-implicit scheme
\begin{equation}
  \frac{\bar{\phi}^{n+1/2}-\phi^n}{\Delta t/2}=\opG\left(\opL\bar{\phi}^{n+1/2}+U[\phi^n]\right),
\end{equation}
which has a local truncation error of $O(\Delta t^2)$.

Just as in the first-order scheme, one can eliminate $\mu^{n+1}$ and $r^{n+1}$ from the second-order schemes \eqref{second} to obtain a linear equation   for $\phi$ similar to \eqref{linear1}, so it can be solved by  using the Sherman--Morrison--Woodbury formula \eqref{SMW} which only involves two linear equations with constant coefficients of the form \eqref{common}.

Regardless of how we obtain $\bar{\phi}^{n+1/2}$, 
multiplying the three equations with $\mu^{n+1/2}$, $(\phi^{n+1}-\phi^n)/\Delta t$, $(r^{n+1}+r^n)/\Delta t$, we derive the following:
\begin{theorem}
  The scheme \eqref{second} is second-order accurate, unconditionally energy stable in the sense that 
 \begin{align*}
  \frac{1}{\Delta t}\Big[\tilde\En^{n+1}-\tilde\En^n\Big]%\nonumber\\
    =-(\mu^{n+1/2},\opG \mu^{n+1/2}),\label{energy_discrete_CN}
\end{align*}
where $\tilde\En$ is the modified energy defined in \eqref{modE}, and one can obtain $(\phi^{n+1},\mu^{n+1}, r^{n+1})$ by solving two linear equations with constant coefficients of the form \eqref{common}.

\end{theorem}

We can also construct semi-implicit  second-order scheme based on BDF formula, which we denote as SAV/BDF:
 \begin{subequations}\label{second_BDF}
\begin{align}
  \frac{3\phi^{n+1}-4\phi^n+\phi^{n-1}}{2\Delta t}=&\opG\mu^{n+1}, \label{second_BDF_1}\\
  \mu^{n+1}=&\opL \phi^{n+1}
  +\frac{r^{n+1}}{\sqrt{\En_1[\bar{\phi}^{n+1}]}}
  U[\bar{\phi}^{n+1}], \label{second_BDF_2}\\
  3r^{n+1}-4r^n+r^{n-1}=&\int_\Omega\frac{U[\bar{\phi}^{n+1}]}{2\sqrt{\En_1[\bar{\phi}^{n+1}]}}
  (3\phi^{n+1}-4\phi^n+\phi^{n-1})\md\bm{x}. \label{second_BDF_3}
\end{align}
\end{subequations}
Here, $\bar{\phi}^{n+1}$ can be any explicit approximation of $\phi(t^{n+1})$ with an error of $O(\Delta t^2)$. 
Multiplying the three equations with $\mu^{n+1}$, $(3\phi^{n+1}-4\phi^n+\phi^{n-1})/\Delta t$, $r^{n+1}/\Delta t$ and integrating the first two equations, and using the identity:
\begin{equation}
\begin{split}
 2(a^{k+1},3a^{k+1}-4a^k+a^{k-1})=&|a^{k+1}|^2+ |2a^{k+1}-a^k|^2+|a^{k+1}-2a^k+a^{k-1}|^2\\
 &-|a^k|^2-|2a^k-a^{k-1}|^2,
 \end{split}
\end{equation}
we obtain the following:
\begin{theorem}
The scheme \eqref{second_BDF} is  second-order accurate, unconditionally energy stable in the sense that 
\begin{align}
  \frac{1}{\Delta t}&\Big\{\tilde \En[(\phi^{n+1}, r^{n+1}), (\phi^n,r^n)]-
  \tilde \En[(\phi^{n}, r^{n}), (\phi^{n-1},r^{n-1})]\big\}
    \nonumber\\
    &+ \frac{1}{\Delta t}\Big\{\frac{1}{4}\big(\phi^{n+1}-2\phi^n+\phi^{n-1},\opL (\phi^{n+1}-2\phi^n+\phi^{n-1})\big)\nonumber\\
    &+\frac{1}{2}(r^{n+1}-2r^n+r^{n-1})^2\Big\}=(\mu^{n+1},\opG\mu^{n+1}),\nonumber%\label{energy_discrete_BDF}
\end{align}
where the modified discrete energy is defined as
\begin{equation*}\label{modE_BDF}
\begin{split}
 \tilde \En[(\phi^{n+1}, r^{n+1}), (\phi^n,r^n)]=&\frac{1}{4}\Big( (\phi^{n+1},\opL \phi^{n+1})+\big(2\phi^{n+1}-\phi^n,\opL(2\phi^{n+1}-\phi^n)\big)\Big)\\
 &+\frac{1}{2}\big((r^{n+1})^2+(2r^{n+1}-r^n)^2\big), 
 \end{split}
\end{equation*}
and one can obtain $(\phi^{n+1},\mu^{n+1}, r^{n+1})$ by solving two linear equations with constant coefficients of the form \eqref{common}.
\end{theorem}

We observe that the modified energy $\tilde \En[(\phi^{n+1}, r^{n+1}, (\phi^n,r^n)]$ is an approximation of the original energy $\En(\phi^{n+1})$ if $(r^{n+1})^2$ is an approximation of $\En_1(\phi^{n+1})$.

\subsection{Gradient flows of multiple functions}
We describe below the SAV approach for gradient flows of multiple functions $\phi_1,\ldots,\phi_k$:
%We will discuss the case where the linear terms are decoupled, that is, 
\begin{align}
  \En[\phi_1,\ldots,\phi_k]=\sum_{i,j=1}^k d_{ij}(\phi_i,\opL\phi_j)+\En_1[\phi_1,\ldots,\phi_k], 
\end{align}
where $\opL$ is a self-adjoint non-negative linear operator, the constant matrix $A=(d_{ij})$ is symmetric positive definite. 
Also we assume that $\En_1\ge C_1>0$. 
A simple case with decoupled linear terms, i.e. $d_{ij}=\delta_{ij}$, is considered in \cite{SXY17}. However, some applications  (cf. for example \cite{Chen02,Boy.M14,Dong14})  involve coupled linear operators  which render the problem very difficult to solve numerically by existing methods.

Denote $U_i=\delta\En_1/\delta\phi_i$, and 
 introduce $r(t)=\sqrt{\En_1}$ as the scalar auxiliary variable. 
 With a dissipation operator $\opG$ for $\phi_i$, the gradient flow is then given by 
\begin{subequations}
\begin{align}
  \frac{\partial\phi_i}{\partial t}=&\opG\mu_i, \\
  \mu_i=&\sum_{j=1}^kd_{ij}\opL\phi_j+\frac{r}{\sqrt{\En_1}}U_i,\\
  r_t=&\frac{1}{2\sqrt{\En_1}}\int_\Omega  U_i\frac{\partial\phi_i}{\partial t} \md\bm{x}.
\end{align} 
\end{subequations}
Taking the inner products of the above three equations with $\mu_i$, $ \frac{\partial\phi_i}{\partial t}$ and $2r$, summing over $i$ and using the facts that $\opL$ is self-adjoint and $d_{ij}=d_{ji}$, we obtain
the energy law:
\begin{equation}
 \frac d{dt}\En[\phi_1,\ldots,\phi_k]=\sum_{i=1}^k (\opG\mu_i,\mu_i).
\end{equation}
Consider, for example, the following  second-order SAV/CN scheme:
\begin{subequations}\label{Msecond} 
\begin{align}
  \frac{\phi_i^{n+1}-\phi_i^n}{\Delta t}=&\opG\mu_i^{n+1/2}, \label{Msecond_1}\\
  \mu_i^{n+1/2}=&\frac 12\sum_{j=1}^kd_{ij}\opL(\phi_j^{n+1}+\phi_j^n)
  +\frac{ U_i[{\bar\phi}_1^{n+1/2},\cdots,{\bar\phi}_k^{n+1/2}]}{2\sqrt{\En_1[\bar{\phi}_1^{n+1/2},\cdots,{\bar\phi}_k^{n+1/2}]}}
 (r^{n+1}+r^n), \label{Msecond_2}\\
  r^{n+1}-r^n=&\int_\Omega\sum_{j=1}^k\frac{U_j[\bar{\phi}_1^{n+1/2},\cdots,\bar{\phi}_k^{n+1/2}]}{2\sqrt{\En_1[\bar{\phi}_1^{n+1/2},\cdots,\bar{\phi}_k^{n+1/2}]}}
  (\phi_j^{n+1}-\phi_j^n)\md\bm{x}, \label{Msecond_3}
\end{align}
\end{subequations}
where $\bar{\phi}_j^{n+1/2}$ can be any second-order explicit approximation of  $\phi(t^{n+1/2})$.
We Multiply the above three equations with $\Delta t\mu_i^{n+1/2}$, $\phi_i^{n+1}-\phi_i^n$, $r^{n+1}+r^n$ and take the sum over $i$. Since $\opL$ is self-adjoint and $d_{ij}=d_{ji}$, we have
\begin{equation*}
 \begin{split}
  \frac 12 \big(\sum_{j=1}^kd_{ij}\opL(\phi_j^{n+1}+\phi_j^n),\phi_i^{n+1}-\phi_i^n\big)
=\frac 12 \sum_{j=1}^kd_{ij}[(\opL\phi_j^{n+1},\phi_i^{n+1})-
 (\opL\phi_j^{n},\phi_i^{n})],
 \end{split}
\end{equation*}
which immediately leads to energy stability. Next, we describe how to implement \eqref{Msecond} efficiently. 

 Denote 
$$
p_i^n=\frac{U_i[\bar{\phi}_1^{n+1/2},\cdots,\bar{\phi}_k^{n+1/2}]}{\sqrt{\En_1[\bar{\phi}_1^{n+1/2},\cdots,\bar{\phi}_k^{n+1/2}]}}, $$
and substitute \eqref{Msecond_2} and \eqref{Msecond_3} into \eqref{Msecond_1}, we can eliminate $\mu_i^{n+1/2} and $ $r^{n+1}$ to obtain a coupled linear  system of $k$ equations of the following form
\begin{align}\label{mlinear}
  \phi_i^{n+1}-\frac{\Delta t}2 \sum_{j=1}^k d_{ij} \opG \opL
  \phi_j^{n+1}-\frac 12\sum_{j=1}^k(\phi_j^{n+1},p_j^n)\opG p_i^n=b_i^n,\quad i=1,\cdots, k,
\end{align}
where $b_i^n$ includes all known terms in the previous  time steps. Let 
us denote 
\begin{equation}
\begin{split}
&\bar\phi^{n+1}=(\phi_1^{n+1},\cdots, \phi_k^{n+1})^T,\quad
\bar b^n=(b_1^n,\cdots,b_k^n)^T, \\
&\bar u=\frac 12(\opG p_1^n,\cdots,\opG p_k^n),\quad
\bar v=(p_1^n,\cdots,p_k^n).
\end{split}
\end{equation}
Let us denote   $D=(\frac{\Delta t}2d_{ij})_{i,j=1,\cdots,k}.$
 The above system can be written in the following matrix form:
\begin{equation}
 ({\cal A}+\bar u \bar v^T)\bar\phi^{n+1}=\bar b^n, 
\end{equation}
where the operator ${\cal A}$ is defined by
\begin{equation}
 {\cal A}\bar\phi^{n+1}=\bar\phi^{n+1}-\opG \opL D\bar\phi^{n+1}.
\end{equation}
 %$$({\cal A}\bar\phi^{n+1})_i=\phi_i^{n+1}-\frac{\Delta t}2 \sum_{j=1}^k d_{ij} \opG \opL  \phi_j^{n+1}:=\phi_i^{n+1}-\opG \opL (B\bar\phi^{n+1})_i,$$
  Hence, by using \eqref{SMW}, the solution of \eqref{mlinear} can be obtained by solving two linear systems of the form  ${\cal A}\bar \phi=\bar b$.
 
It remains to describe how to solve the linear system ${\cal A}\bar \phi=\bar b$ efficiently. 
Since $D$ is symmetric positive definite, we can compute an orthonormal eigen-matrix $E$ such that $D=E\Lambda E^T$ where $\Lambda=\text{diag}(\lambda_1,\cdots, \lambda_k)$. Setting $\bar \psi=E^T \bar\phi$, we  have 
$${\cal A}\bar\phi=\bar\phi-\opG \opL E\Lambda E^T \bar\phi= E(I-\opG \opL \Lambda)\bar \psi.$$
Hence, ${\cal A}\bar \phi=\bar b$  decouples into a sequence of elliptic equations:
\begin{equation}\label{phii}
 \psi_i-\lambda_i \opG \opL \psi_i=(E^T\bar b)_i, \quad i=1,\cdots,k.
\end{equation}
 To summarize,  ${\cal A}\bar \phi=\bar b$ can be efficiently solved as follows:
\begin{itemize}
\item Compute the eigen-pair $(E,\Lambda)$;
\item Compute $E^T\bar b$;
\item Solve the decoupled equations  \eqref{phii};
\item Finally, the solution is: $\bar\phi=E\bar\psi$.
\end{itemize}
In summary:
\begin{theorem}
The scheme \eqref{Msecond} is second-order accurate, unconditionally energy stable in the sense that  
\begin{equation*}
\begin{split}
 \frac 1{2\Delta t}& \big[\sum_{j=1}^kd_{ij}(\opL\phi_j^{n+1},\phi_i^{n+1})+(r^{n+1})^2\big]-
 \frac 1{2\Delta t} \big[\sum_{j=1}^kd_{ij}(\opL\phi_j^{n},\phi_i^{n})+(r^n)^2\big]=\sum_{i=1}^k (\opG\mu_i,\mu_i),
 \end{split}
\end{equation*}
and one can obtain $r^{n+1}$  and $(\phi_j^{n+1},\mu_j^{n+1})_{1\le j\le k}$  by solving two sequences of decoupled linear equations with constant coefficients of the form \eqref{phii}.
 \end{theorem}

\subsection{Full discretization}
To simplify the presentation,  we have only discussed the time discretization in the above. However, since  the stability proofs of SAV schemes are all variational,  they can be straightforwardly extended to full discrete SAV schemes with Galerkin finite element methods or Garlerkin spectral methods or even finite difference methods with summation by parts.

\section{Numerical validation\label{Num}}
In this section, we apply the SAV/CN and SAV/BDF schemes to several gradient flows
to demonstrate the efficiency and accuracy  of the SAV approach. In all  examples, we assume  periodic boundary conditions and use a Fourier-spectral method for space variables.% although other suitable boundary conditions can be treated similarly. 

\subsection{Allen-Cahn,  Cahn-Hilliard and fractional Cahn-Hilliard  equations} 
The Allen-Cahn \cite{allen1979microscopic} and Cahn-Hilliard equations \cite{cahn1958free,cahn1959free}, %phase field model is built on the free energy for nonuniform systems 
are widely used in the study of interfacial dynamics \cite{allen1979microscopic, rogers1989numerical,anderson1998diffuse,gurtin1996two,kim2012phase,liu2003phase,lowengrub1998quasi,yue2004diffuse,Ain.M17}. %(see for example \cite{talanquer1996nucleation}). 
They are built with the free energy
\begin{equation}
  \En[\phi]=\frac{1}{2}\int\left( |\nabla \phi|^2+\frac{1}{4\epsilon^2}(1-\phi^2)^2\right) \md \bm{x}. \label{energy_interface}
\end{equation}
%Here  $\phi$ is an order parameter that represents the volume fraction of two components. The gradient term represents the interfacial energy, the double-well bulk energy gives the equilibrium values $\phi=\pm 1$ of each component in homogeneous system, and the parameter $\epsilon$ determines the thickness of interface. 
We consider the $H^{-s}$ gradient flow, which leads to the fractional Cahn--Hilliard equation: 
\begin{equation}
  \frac{\partial \phi}{\partial t}=-\gamma(-\Delta)^{s}(-\Delta \phi-\frac{1}{\epsilon^2}\phi(1-\phi^2)),\quad 0\le s\le 1. \label{BinMix}
\end{equation}
Here, the fractional Laplacian operator $(-\Delta)^{s}$ is defined via Fourier expansion. More precisely, if $\Omega=(0,2\pi)^2$, then 
we can express $u\in L^2(\Omega)$  as 
$$
u=\sum_{man}\hat{u}_{mn}e^{imx+iny},
$$
so the fractional Laplacian is defined as
$$
(-\Delta)^{s}u=\sum(m^2+n^2)^s\hat{u}_{mn}e^{imx+iny},
$$
When $s=0$ ($L^2$ gradient flow), \eqref{BinMix} is the standard Allen--Cahn equation; when $s=1$, it becomes the standard Cahn--Hilliard equation. 

\iffalse
The $L^2$ gradient flow is the Allen--Cahn equation, 
\begin{equation}
  \frac{\partial \phi}{\partial t}=\gamma(\Delta \phi+\frac{1}{\epsilon^2}\phi(1-\phi^2)); 
\end{equation}
and the $H^{-1}$ gradient flow is the Cahn--Hilliard equation, 
\begin{equation}
  \frac{\partial \phi}{\partial t}=\gamma\Delta(-\Delta \phi-\frac{1}{\epsilon^2}\phi(1-\phi^2)); 
\end{equation}
\fi

%?? Are any numerical results with fractional $s$?? 
%Yes. Example 4. 

To apply our schemes \eqref{second} or \eqref{second_BDF}  to \eqref{BinMix}, we specify the operators $\opL$, $\opG$ and the energy $\En_1$ as 
\begin{equation}
  \opL=-\Delta+\frac{\beta}{\epsilon^2},\ \opG=-(-\Delta)^s,\ \En_1=\frac{1}{4\epsilon^2}\int_\Omega(\phi^2-1-\beta)^2 \md\bm{x}, \label{BMixOp}
\end{equation}
where $\beta$ is a positive number to be chosen. Then we have 
$$
U[\phi]=\frac{\delta \En_1}{\delta \phi}=\frac{1}{\epsilon^2}\phi(\phi^2-1-\beta). 
$$

\begin{example}\label{ex1}
(Convergence rate of SAV/CN scheme for the standard Cahn-Hilliard equation) 
We choose the computational domain as $[0,2\pi]^2$, $\epsilon=0.1$, and $\gamma=1$. 
The initial data is chosen as smooth one $u_0(x,y)=0.05\sin(x)\sin(y)$.
%We choose the one-dimensional domain $[0,2\pi]$, $\epsilon=0.1$ and $\beta=1$. The initial value is $\phi(x,0)=0.05\cos(x)$. 
\end{example}

%The first part of this example intend to verify the numerical convergence rate of SAVT/CN scheme. We take the initial value as $u(x,0)=0.05\cos(x)$. We check the convergence rate of SAVT/CN \eqref{second_1}-\eqref{second_3} with $\beta=1$.
%We use the finite difference method for spatial discretization with mesh size $h={2\pi}/{2^{10}}$ so that the spatial discretization error can be ignored when compared with the time discretization error.
We use the Fourier Galerkin method for spatial discretization with $N=2^7$, and choose $\beta=1$. 
To compute a reference solution, we use the fourth-order exponential time differencing Runge-Kutta method (ETDRK4)\footnote{Although ETDRK4 has higher order of accuracy, it does not guarantees energy stability, and the implementation can be difficult since it requires to calculate matrix exponential.} \cite{cox2002exponential} with $\Delta t$ sufficiently small. 
%We test the numerical convergence rates for the SAVT/CN scheme and the SAVT/BDF scheme for the Cahn--Hilliard equation. 
%Again, the reference solution is obtained by ETDRK4. 
The numerical errors at $t=0.032$ for SAV/CN and SAV/BDF are shown in TABLE \ref{tab1}, where we can observe the second-order convergence for both schemes. 
%??What is $\beta$ used here??

%The convergence rates for the Allen--Cahn equation are plotted in  FIG. \ref{acrate}. We observe that  SAV/CN scheme achieves the second-order convergence rate at two randomly choose time $T=0.8$ and $T=4$. 
\iffalse
\begin{figure}%[!h]
\centering
\includegraphics[width=1.35\textwidth,keepaspectratio]{err.pdf}\vspace{-1.2cm}
\caption{(Example \ref{ex1}) Numerical convergence rate of SAV/CN scheme. The triangle represents the second-order convergence rate. }\label{acrate}
\end{figure}
\fi

%The numerical errors at $T=1$ for the Cahn--Hilliard equation are shown in TABLE \ref{tab1}, where we also observe the second-order convergence for both schemes. 
\iffalse
Here we should point out one implementation technique. Since both two schemes are three-level, how to prepare the first step solution $u^1$ will become a concern.
Some references just take $u^1=u^0$ for multistep methods. If we also follow this technique, we will only get first order convergence rate. However, if we compute $u^1$ just by first order semi-implicit scheme, we can recover the  second order convergence rate.
\fi

\begin{table}%[h]
  \centering
  \begin{tabular}{|c|c|c|c|c|c|c|}    \hline
   Scheme& &$\Delta t$=1.6e-4 &$\Delta t$=8e-5 &$\Delta t$=4e-5 &$\Delta t$=2e-5&$\Delta t$=1e-5\\
    \hline
    \multirow{2}{*}{SAV/CN}
    &Error & 1.74e-7 & 4.54e-8 & 1.17e-8 & 2.94e-9 &2.01e-10\\
     \cline{2-7}
    &Rate & - & 1.93 & 1.96 & 1.99 & 2.01\\
    \hline
    \multirow{2}{*}{SAV/BDF}
    &Error & 1.38e-6 & 3.72e-7 & 9.63e-8 & 2.43e-8 & 5.98e-9\\
     \cline{2-7}
    &Rate & - & 1.89 & 1.95 & 1.99  & 2.02\\
    \hline
  \end{tabular}
  \caption{(Example \ref{ex1}) Errors and convergence rates of SAV/CN and SAV/BDF  scheme for the Cahn--Hilliard equation.}\label{tab1}
\end{table}

\begin{example}\label{ex1_3}
We solve a benchmark problem for the Allen--Cahn equation (see \cite{Chen1998Applications}). 
Consider a two-dimensional domain $(-128,128)^2$ with  a circle of radius $R_0=100$. 
In other words, the initial condition is given by 
\begin{equation}
  \phi(x,y,0)=\left\{
  \begin{array}{ll}
    1,& x^2+y^2<100^2, \\
    0,& x^2+y^2\ge 100^2. 
  \end{array}
  \right.
\end{equation}
By mapping the domain to $(-1,1)^2$, the parameters in the Allen-Cahn  equation are given by 
$\gamma=6.10351\times 10^{-5}$ and $\epsilon=0.0078$. 
\end{example}
In the sharp interface limit ($\epsilon\to 0$, which is suitable because the chosen $\epsilon$ is small), the radius at the time $t$ is given by 
\begin{equation}
  R=\sqrt{R_0^2-2t}.
\end{equation}
We use the Fourier Galerkin method to express $\phi$ as 
\begin{equation}
\phi=\sum_{n_1,n_2\le N}\hat{\phi}_{n_1n_2}e^{i\pi (n_1x+n_2y)}, 
\end{equation}
with $N=2^9$. 
We choose $\beta=0.1$ and let the time step $\Delta t$ vary. %, $\Delta t=0.5,\,0.2,\,0.1,\,0.05,\,0.02,\,0.01$. 
The computed radius $R(t)$ using the  SAV/CN scheme is plotted in FIG. \ref{radius}. 
We observe that $R(t)$ keeps monotonously decreasing and very close to the sharp interface limit value, even when we choose a relatively large $\Delta t$. 
In \cite{shen2010numerical} this benchmark problem is solved using different stabilization methods. Our result proves to be much better than the result in that work, where the oscillation around the limit value is apparent, even if the time step has been reduced to $\Delta t=10^{-3}$. 
We also plot the original energy and the modified energy $(\phi^n, \opL\phi^n)+ (r^n)^2$ in FIG. \ref{radius} for $\Delta t=0.5$, and find that they are very close. 
\begin{figure}
\centering
\includegraphics[width=.49\textwidth,keepaspectratio]{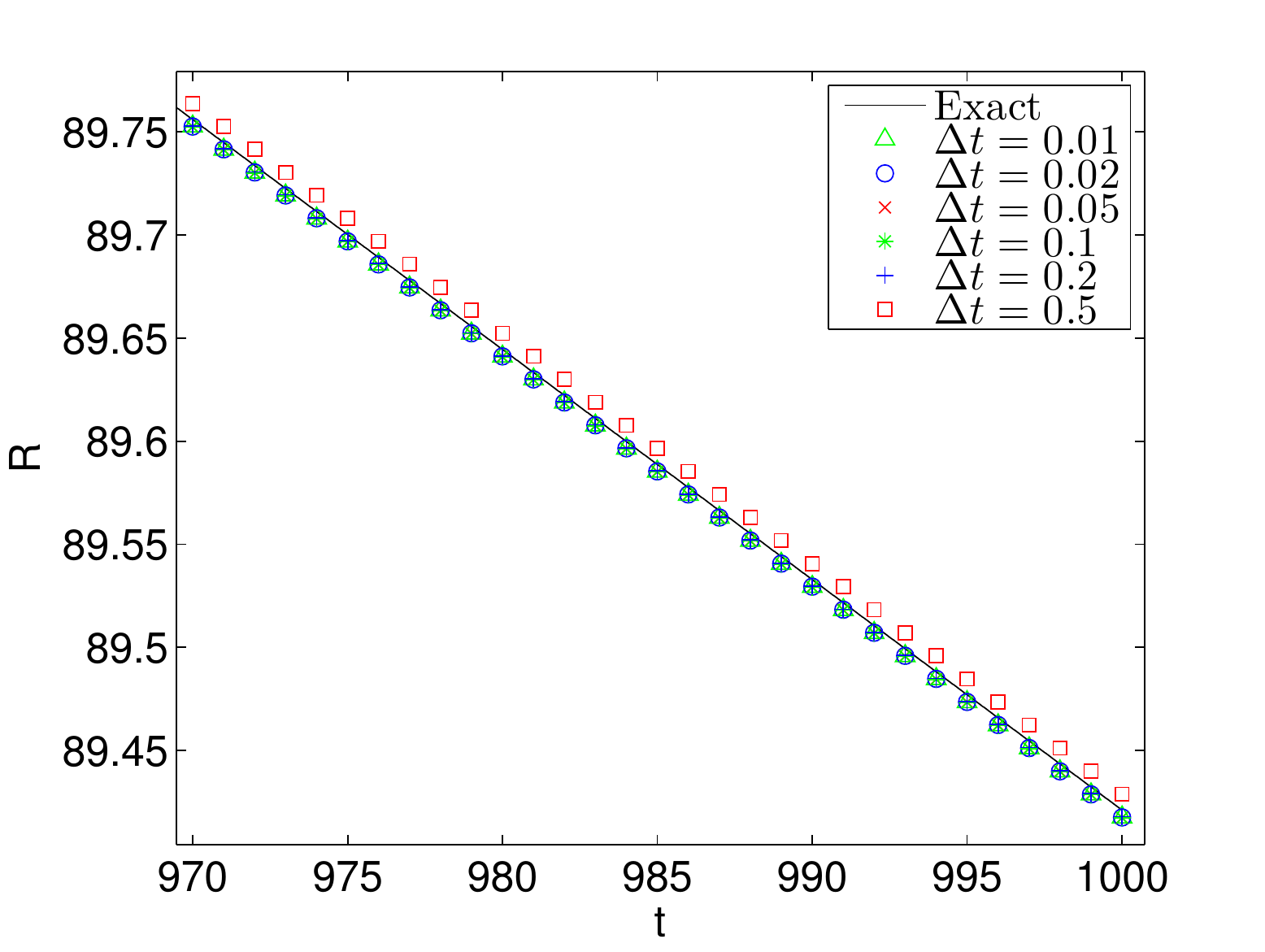}
\includegraphics[width=.49\textwidth,keepaspectratio]{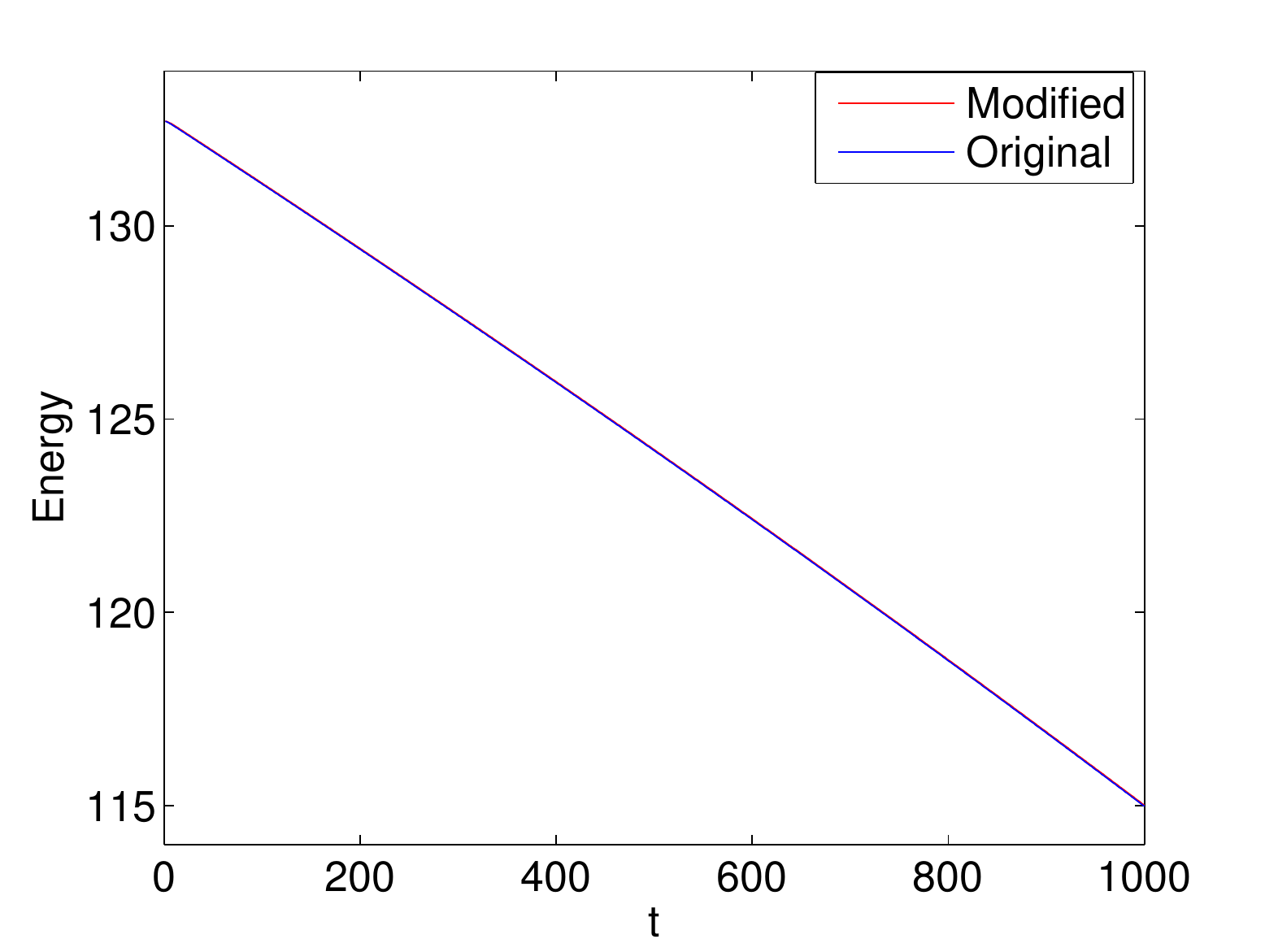}
\caption{(Example \ref{ex1_3}) The evolution of radius $R(t)$ and the free energy (both original and modified). For the free energy, $\Delta t=0.5$.}
\label{radius}
\end{figure}

\begin{example}\label{ex1_2}
(Comparison of SAV/CN and IEQ/CN schemes for the Allen--Cahn equation in 1D) 
The parameters are the same as the first example. 
The domain is chosen as $[0,2\pi]$, discretized by the finite difference method with $N=2^{10}$. 
The initial condition $\phi(x,0)$ is now a randomly generated function. The reference solution is also obtained using ETDRK4. 
\end{example}

\begin{comment}
We compare the SAV/CN scheme with the IEQ/CN scheme, given as follows, 
\begin{align}
  \frac{\phi^{n+1}-\phi^n}{\Delta t}=&-\opL\frac{1}{2}(\phi^{n+1}+\phi^n)
  +\frac{q^{n+1}+q^n}{2\sqrt{g(\bar{\phi}^{n+1/2})}}
  U[\bar{\phi}^{n+1/2}], \label{Vsecond_1}\\
  q^{n+1}-q^n=&\frac{U[\bar{\phi}^{n+1/2}]}{2\sqrt{g(\bar{\phi}^{n+1/2})}}
  (\phi^{n+1}-\phi^n). \label{Vsecond_2}
\end{align}

We compare the result of SAV/CN and IEQ/CN  schemes with the same randomly generated initial value $\phi(x,0)$ at $T=0.1$ and $T=1$. 
\end{comment}
We plot the numerical results at $T=0.1$ and $T=1$ by  SAV/CN and IEQ/CN  schemes in FIG. \ref{accomp}.   We used two different time steps $\Delta t=10^{-4},\ 10^{-3}$.  We observe that with $\Delta t=10^{-4}$, %(shown in the first row),  
both SAV/CN scheme and IEQ/CN scheme agree  well with the reference solution. 
However, with $\Delta t=10^{-3}$, the solution by SAV/CN scheme still
agree well with the reference solution at both
$T=0.1$ and $T=1$, while
the solution obtained by SAV/CN scheme has visible differences with  the reference solution, and  
 violates the maximum principle $|\phi|\le 1$.  
This example  clearly indicates that the SAV/CN scheme is more accurate than  the IEQ/CN scheme, in addition to its easy implementation. 
%{\color{blue}Add plot of original and modified energy in FIG. \ref{accomp}. }
\begin{figure}%[!h]
\centering
\includegraphics[width=1.05\textwidth,keepaspectratio]{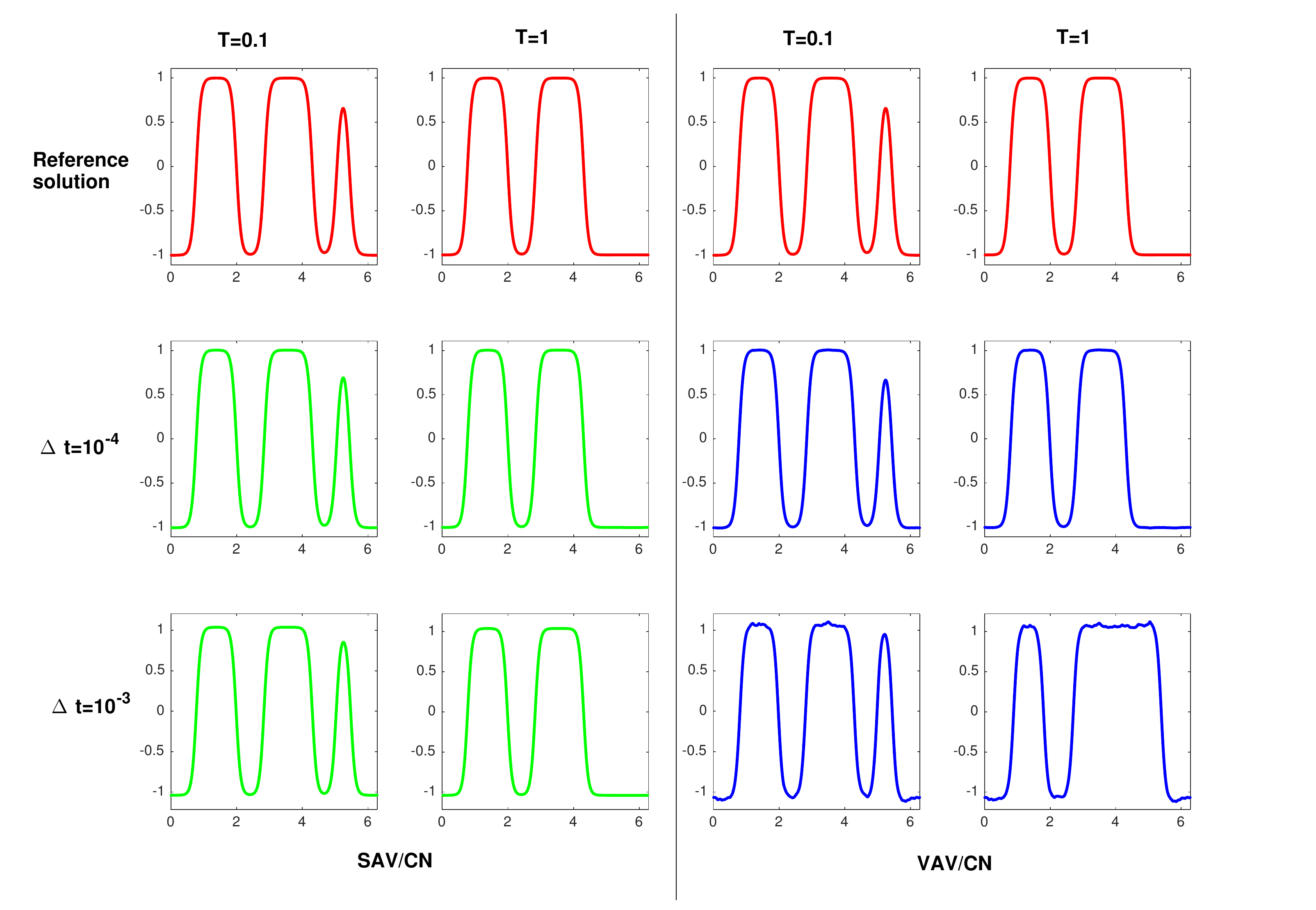}\vspace{-2em}
\caption{(Example \ref{ex1_2})  Comparison of SAV/CN  and IEQ/CN  schemes. }\label{accomp}
\end{figure}

\iffalse
In this part, we consider two-dimensional standard or fractional Cahn--Hilliard equations. We still start with 2D version Ginzburg-Landau free energy functional \eqref{GLE}, and we use the same technique in \eqref{mdfen} to rearrange the energy functional. But here we consider a different energy dissipation gradient flow, namely $H^{-\alpha}$ gradient flow
\begin{equation}\label{eqch}
u_t=-(-\Delta)^{\alpha}\left(\left(-\Delta u+\frac{\beta}{\epsilon^2}u\right)+\frac{1}{\epsilon^2}u(u^2-1-\beta)\right),
\end{equation}
with $\opG=-(-\Delta)^{\alpha}$, $\opL=-\Delta+\frac{\beta}{\epsilon^2}I$, and $f(u)=\frac{1}{\epsilon^2}u(u^2-1-\beta)$. The computational domain is fixed as $\Omega=[0,2\pi]^2$, and the fractional Laplacian operator $(-\Delta)^{\alpha}$ is defined via Fourier decomposition. Namely, for any $u\in L^2_{per}(\Omega)$,
$$u=\sum_{m,n\in Z}\hat{u}_{mn}e^{imx+iny},$$
the fractional Laplacian is then defined by
$$(-\Delta)^{\alpha}u=\sum_{m,n\in Z}(m^2+n^2)^\alpha\hat{u}_{mn}e^{imx+iny},$$
which becomes the exact standard Laplace operator when $\alpha=1$. Actually, the equation \eqref{eqch} becomes the standard Cahn--Hilliard equation for $\alpha=1$. On the other hand, it becomes the standard Allen--Cahn equation when $\alpha=0$.
\fi

\begin{example}\label{ex3} We examine the effect of fractional dissipation mechanism on the phase separation and coarsening process.
Consider the fractional Cahn--Hilliard equation in $[0,2\pi]^2$. 
We fix $\epsilon=0.04$ and take $s$ to be $0.1$, $0.5$, $1$, respectively. 
We use the Fourier Galerkin method with $N=2^7$, and the time step $\Delta t=8\times 10^{-6}$. 
The initial value is the sum of a randomly generated function $\phi_0(x,y)$ with  the average of $\phi$:
$$
\bar{\phi}=\frac{1}{4\pi^2}\int_{0\le x,y \le 2\pi}\,\phi \,\md x\md y,
$$
 chosen as $0.25,\;0,\;-0.25$, respectively.
\end{example}
\begin{figure}%[!h]
\centering
\includegraphics[width=0.9\textwidth,height=12cm]{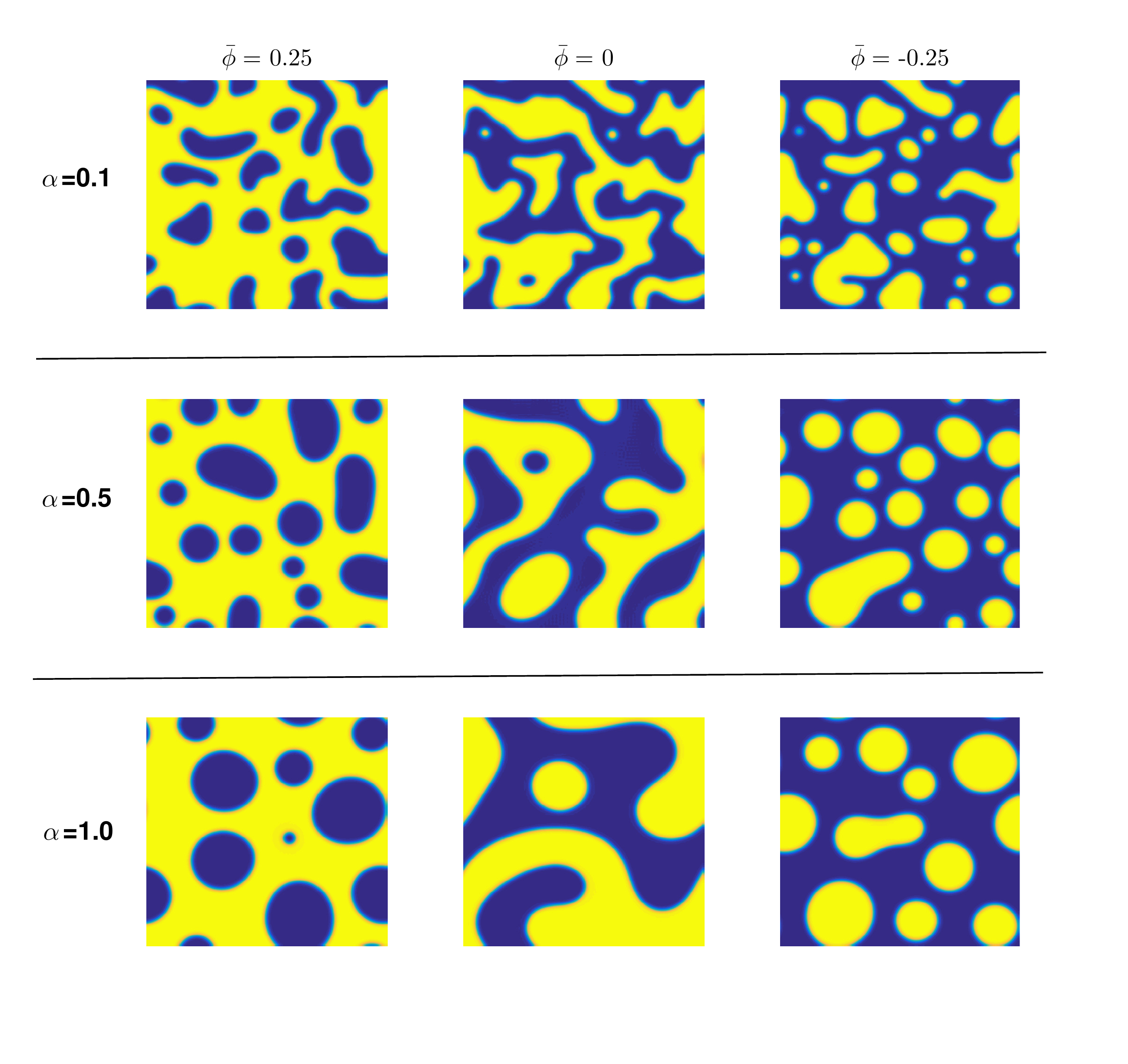}\vspace{-1cm}
\caption{(Example \ref{ex3}) Configurations at time $T = 0.032$ with random initial
condition for different values of fractional order $\alpha$ and means $\bar{\phi}$.}\label{fch}
\end{figure}

We used the SAV/BDF scheme to compute the configuration at $T=0.032$, 
which is shown in FIG. \ref{fch}. 
We observe that regardless of $\bar{\phi}$, when $\alpha$ is smaller, the phase separation and coarsening process is slower, which is consistent with the results in \cite{Ain.M17}. 

\subsection{Phase field crystals}
We now consider gradient flows of $\phi(\bm{x})$ that describes modulated structures. 
Free energy of this kind was first found in Brazovskii's work \cite{brazovskii1975phase}, known as the Landau-Brazovskii model. 
Since then, the free energy, including many variants, has been adopted to study various physical systems (see for example \cite{garel1982phase,andelman1987phase,gompper1990correlation,xu2017computing}). 
A usual free energy takes the   form, 
\begin{equation}
  \En (\phi)=\int_{\Omega}\left\{\frac{1}{4}\phi^4+\frac{1-\epsilon}{2}\phi^2-|\nabla \phi|^2+\frac{1}{2}(\Delta \phi)^2\right\}\md\bm{x}, \label{LB} 
\end{equation}
subjected to a constraint that the average $\bar{\phi}$ remains to  be a constant. 
This constraint can be automatically satisfied with  
an $H^{-1}$ gradient flow, which is also referred to as phase field crystals model because it is widely adopted in the dynamics of crystallization \cite{elder2002modeling,elder2004modeling,elder2007phase}. 
\begin{comment}
All the terms in the Landau--Brazovskii model come from the approximation of the virial expansion up to the fourth order. Generally, the second order term in the virial expansion is given by 
$$
F_2[\phi]=\frac{1}{2}\int_\OmegaK(\bm{x}-\bm{x'})\phi(\bm{x})\phi(\bm{x'})\md\bm{x}\md\bm{x'}. 
$$
The gradient terms are obtained by approximating the kernel function $K$ in the Fourier space. The approximation involves a delta function, which actually let $K$ localized. 
In some recent works \cite{du2012analysis,silling2000reformulation}, the kernel $K$ is chosen as functions with a small support, which allows interactions within a small distance. Models of this kind are usually referred to as nonlocal models. 

\textit{Remark.} For general nonlocal models, the presence of multiple integrals (there may also be triple or quadruple integrals, etc.) will make the IEQ schemes not applicable, but does not affect the usage of SAV schemes. 
This is also a potential advantage of SAV over IEQ. 
\end{comment}
To demonstrate the flexibility of SAV approach, we will focus on a free energy with a nonlocal kernel. 
Specifically, we replace the Laplacian by a nonlocal linear operator $\opL_{\delta}$ \cite{silling2000}: 
$$
\mathcal{L}_\delta \phi(\bm{x})=\int_{\mathcal{B}(\bm{x},\delta)} \rho_\delta(|\bm{y}-\bm{x}|)\big(\phi(\bm{y})-\phi(\bm{x})\big)\md\bm{y}, 
$$
leading to the free energy, 
\begin{equation}
  \En (\phi)=\int_{\Omega}\left\{\frac{1}{4}\phi^4+\frac{1-\epsilon}{2}\phi^2+\phi\opL_{\delta} \phi+\frac{1}{2}(\opL_{\delta} \phi)^2\right\}\md \bm{x}. \label{nonlocal0} 
\end{equation}
Let the dissipation mechanism be given by $\opG=\opL_{\delta}$. 
Then we obtain the following gradient flow, 
\begin{equation}
\frac{\partial \phi}{\partial t}=\mathcal{L}_\delta(\mathcal{L}_\delta^2\phi+2\mathcal{L}_\delta \phi+(1-\epsilon)\phi+\phi^3). \label{nonlocal1}
\end{equation}
For the above problem, it is difficult to solve the linear system resulted from the IEQ approach, but it can be easily implemented with the SAV approach. 

Let $\Omega$ be a rectangular domain $[0,2\pi)^2$ with periodic boundary conditions, the eigenvalues of $\opL$ can be expressed explicitly. 
In fact, it is easy to check that for any integers $m$ and $n$, $e^{imx+iny}$ is an eigenfunction of $\mathcal{L}_\delta$, and the corresponding eigenvalue is given by 
$$
\lambda_\delta(m,n)=\int_{0}^{\delta}r\rho_\delta(r)\int_{0}^{2\pi}\left(\cos\left(r
\left(m\cos\theta+n\cos\theta\right)\right)-1\right)\md\theta \md r,
$$
which can be evaluated efficiently using a hybrid algorithm \cite{du2017fast}. 
We choose 
$$
\rho_{\delta}(|\bm{x}-\bm{x'}|)
=c_1\frac{2(4-\alpha_1)}{\pi}\frac{1}{\delta^{4-\alpha_1}r^{\alpha_1}}-c_2\frac{2(4-\alpha_2)}{\pi}\frac{1}{\delta^{4-\alpha_2}r^{\alpha_2}}, 
$$
with $c_1=20,\;c_2=19$, $\alpha_1=3,\;\alpha_2=0$ and $\delta=2$. 
Numerical results indicate that  all eigenvalues are  negative, which ensures the nonlocal operator $\opL_\delta$ is negative-semidefinite. 

\begin{figure}%[!h]
  \begin{center}
    %$\mbox{}$\hspace{-0.8cm}
    \includegraphics[width=0.9\textwidth,keepaspectratio]{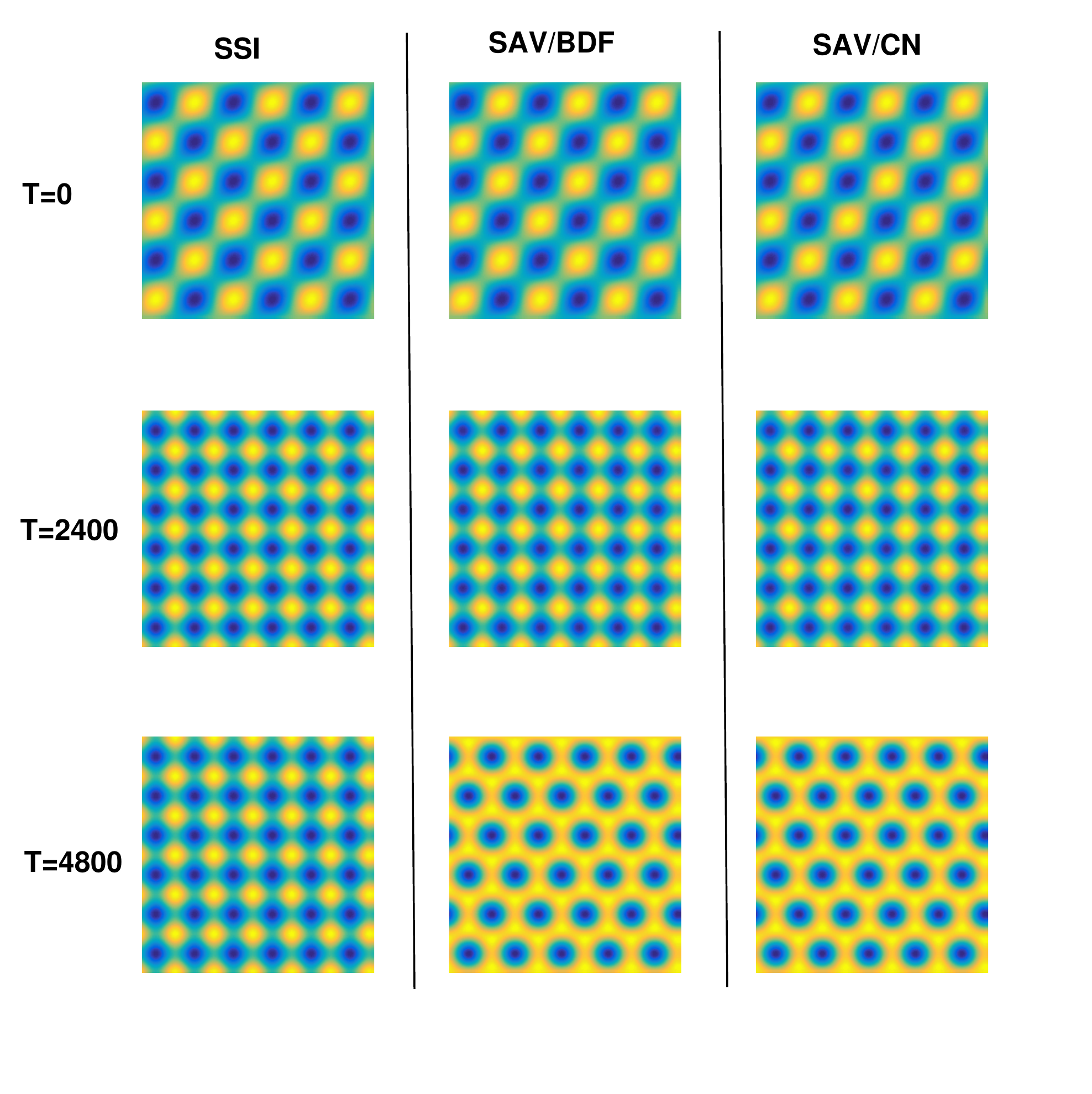}
  \end{center}
  \vspace{-2em}
  \caption{(Example \ref{npfc}) Configuration evolutions for NPFC models by three schemes.} \label{fignpfc}
\end{figure}
\begin{figure}%[!h]
  \begin{center}
    %$\mbox{}$\hspace{-0.8cm}
    \includegraphics[width=15cm]{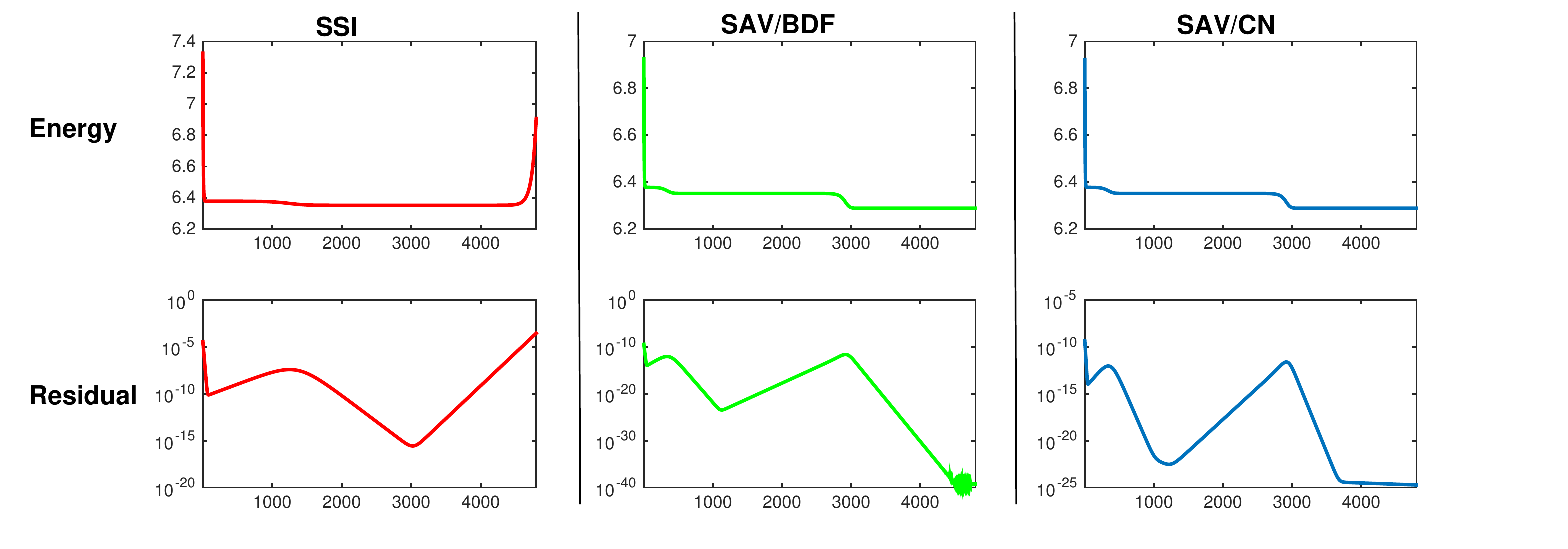}
  \end{center}
  \vspace{-1em}
  \caption{(Example \ref{npfc}) Energy evolutions, and residual evolutions for NPFC models by three schemes.} \label{fignpfcer}
\end{figure}
We applied the SAV/CN and SAV/BDF schemes to \eqref{nonlocal1}. As a  comparison, we also implemented  the following stabilized semi-implicit (SSI) scheme used in \cite{cheng2008efficient}:
\begin{eqnarray}
\displaystyle\nonumber
\frac{\phi^{n+1}-\phi^n}{\Delta t}&=&(1-\epsilon)\mathcal{L}_\delta \phi^{n+1}+2\mathcal{L}_\delta^2\phi^{n+1}+\mathcal{L}_\delta^3\phi^{n+1}+(\phi^n)^3\\\nonumber
\displaystyle
&&+a_1(1-\epsilon)\mathcal{L}_\delta(\phi^{n+1}-\phi^n)-2a_2\mathcal{L}_\delta^2(\phi^{n+1}-\phi^n)+a_3\mathcal{L}_\delta^3(\phi^{n+1}-\phi^n).
\end{eqnarray}
Specifically, we choose $a_1=0$, $a_2=1$ and $a_3=0$ which satisfy the parameters constraints provided in \cite{cheng2008efficient}. 
\begin{comment}
\begin{equation}\label{3a}
a_1<\frac{1}{2}-\frac{3\bar{\phi}^2}{2(1-\epsilon)},\quad a_2\ge \frac{1}{2},\quad a_3\le \frac{1}{2},
\end{equation}
where $\bar{\phi}$ is the volume average of $\phi$. 
\end{comment}

For the SAV schemes, we specify the linear non-negative operator as $\opL=\mathcal{L}_\delta^2+2\mathcal{L}_\delta+I$. The time step is fixed at $\Delta t=1$.
\begin{example}\label{npfc}
  We consider \eqref{nonlocal1} in the two-dimensional domain $[0, 50]\times[0,50]$ with periodic boundary conditions. 
  Fix $\epsilon=0.025$ and $\bar{\phi}=0.07$. 
  The Fourier Galerkin methods is used for spatial discretization with $N=2^7$. 
\end{example}

%The results are presented in FIG. \ref{fignpfc}.
The initial value possesses a square structure, drawn in the first row in FIG. \ref{fignpfc}, and the configurations at $T=2400$ and $4800$ are shown in the other two rows. There is no visible difference between the results for all three schemes at $T=2400$. 
However, for both SAV schemes, the system eventually evolves to a stable hexagonal structure, while for the SSI scheme it remains to be the unstable square structure. 
We also plot the free energy and residual as functions of time for the three schemes (see FIG. \ref{fignpfcer}). 
For the SSI scheme, the residue started to increase when $T>3000$, and the free energy eventually increases, violating the energy law. 
On the other hand, the free energy curves for both SAV schemes remain to be dissipative, with  no visible difference between them. This example clearly shows that our SAV schemes have  much better stability and accuracy  than the  SSI scheme for the nonlocal model \eqref{nonlocal1}. 

\section{Higher order SAV schemes and adaptive time stepping}
We describe below how to construct higher order schemes for gradient flows by combining the SAV approach with higher order BDF schemes, and how to implement adaptive time stepping to further increase the computational efficiency.
\subsection{Higher order SAV schemes}
For the reformulated system \eqref{SAVeq_r}-\eqref{SAVeq_phi}, we can easily use the SAV approach to construct BDF-$k$ ($k\ge 3$) schemes. Since BDF-$k$ ($k\ge 3$) schemes are not A-stable for ODEs, they will not be unconditionally stable.  We will focus on BDF3 and BDF4 schemes below,  as for $k>4$, the resulting BDF-$k$ schemes do not appear to be stable. 
 
The SAV/BDF3 scheme is given by 
\begin{subequations}\label{BDF3}
\begin{align}
 & \frac{11\phi^{n+1}-18\phi^n+9\phi^{n-1}-2\phi^{n-2}}{6\Delta t}=\opG\mu^{n+1},\nonumber\\
 & \mu^{n+1}=\opL \phi^{n+1}
  +\frac{r^{n+1}}{\sqrt{\En_1[\bar{\phi}^{n+1}]}}
  U[\bar{\phi}^{n+1}],\nonumber\\
 & 11r^{n+1}-18r^n+9r^{n-1}-2r^{n-2}=\nonumber\\
\hskip 3pt & \int_\Omega\frac{U[\bar{\phi}^{n+1}]}{2\sqrt{\En_1[\bar{\phi}^{n+1}]}}
  (11\phi^{n+1}-18\phi^n+9\phi^{n-1}-2\phi^{n-2})\md\bm{x},\nonumber
\end{align}
\end{subequations}
where $\bar{\phi}^{n+1}$ is a third-order explicit approximation to $\phi(t_{n+1})$.
 The SAV/BDF4 scheme is given by 
 \begin{subequations}\label{BDF4}
\begin{align}
&  \frac{25\phi^{n+1}-48\phi^n+36\phi^{n-1}-16\phi^{n-2}+3\phi^{n-3}}{12\Delta t}=\opG\mu^{n+1},\nonumber\\
 & \mu^{n+1}=\opL \phi^{n+1}
  +\frac{r^{n+1}}{\sqrt{\En_1[\bar{\phi}^{n+1}]}}
  U[\bar{\phi}^{n+1}],\nonumber\\
& 25r^{n+1}-48r^n+36r^{n-1}-16r^{n-2}+3r^{n-3}=&\nonumber\\
& \int_\Omega\frac{U[\bar{\phi}^{n+1}]}{2\sqrt{\En_1[\bar{\phi}^{n+1}]}}
  (25\phi^{n+1}-48\phi^n+36\phi^{n-1}-16\phi^{n-2}+3\phi^{n-3})\md\bm{x}, \nonumber
\end{align}
\end{subequations}
where $\bar{\phi}^{n+1}$ is a fourth-order explicit approximation to $\phi(t_{n+1})$.
%\begin{comment}

To obtain $\bar{\phi}^{n+1}$ in BDF3, we can use  the extrapolation 
 (BDF3A):
$$\bar{\phi}^{n+1}=3{\phi}^{n}-3{\phi}^{n-1}+{\phi}^{n-2},$$
or  prediction by one BDF2 step (BDF3B):
$$\bar{\phi}^{n+1}=\text{BDF2}\{\phi^{n},{\phi}^{n-1},\Delta t\}.$$
Similarly, to  get $\bar{\phi}^{n+1}$ in BDF4, we can do   the extrapolation  (BDF4A):
$$\bar{\phi}^{n+1}=4{\phi}^{n}-6{\phi}^{n-1}+4{\phi}^{n-2}-\phi^{n-3},$$
or prediction with one step of BDF3A (BDF4B):
$$\bar{\phi}^{n+1}=\text{BDF3}\{\phi^{n},{\phi}^{n-1},{\phi}^{n-2},\Delta t\}.$$
It is noticed that using the  prediction with  a lower order BDF step will double the total computation cost. 
%\end{comment}

%This example is designed to achieve three goals, difference between extrapolation technique and BDF prediction technique for BDF3 and BDF4, numerical convergence, and numerical performances of BDF2, BDF3 and BDF4. 

\begin{example}\label{ex34}
We take Cahn--Hilliard equation as an example to demonstrate the numerical performances of SAV/BDF3 and SAV/BDF4 schemes.  
We fix the computational domain as $[0,2\pi)^2$ and $\epsilon=0.1$. 
We use the Fourier Galerkin method for spatial discretization with $N=2^7$. 
The initial data is  $u_0(x,y)=0.05\sin(x)\sin(y)$.
\end{example}
\begin{figure}[h]
  \begin{center}
    \includegraphics[width=1.05\textwidth]{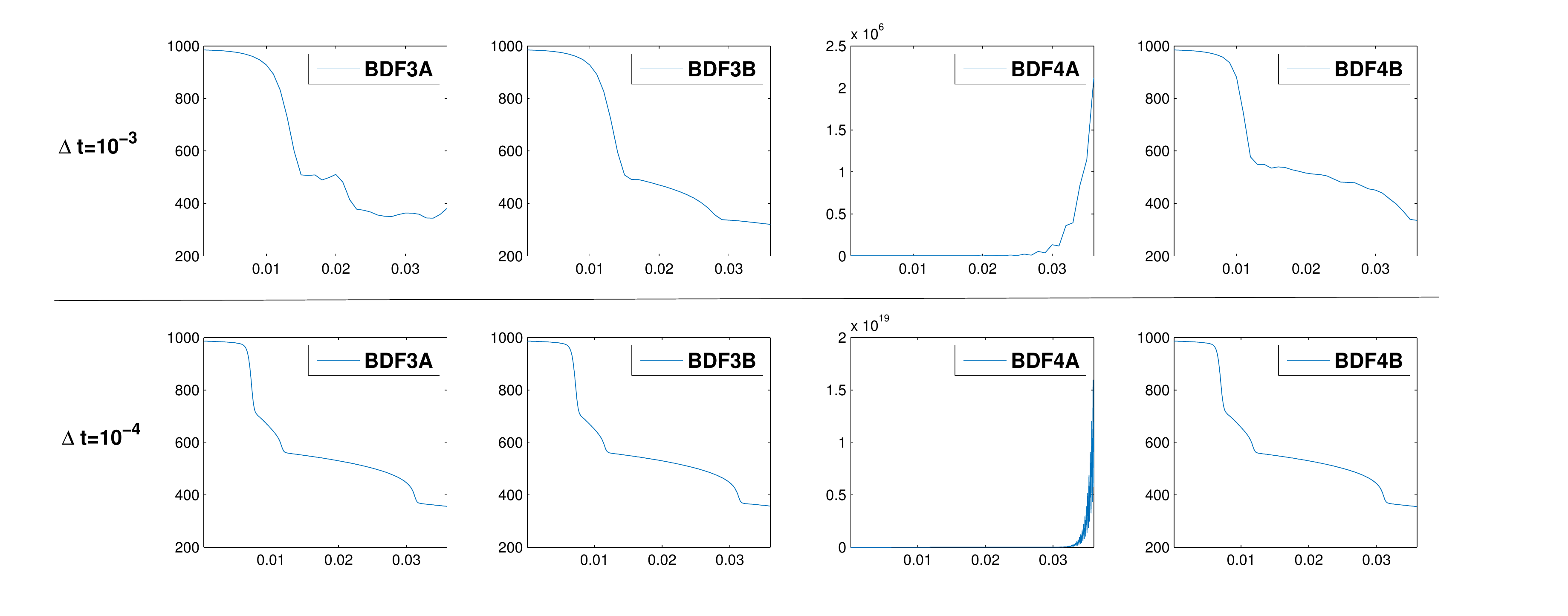}
  \end{center} \vspace{-1em}
  \caption{(Example \ref{ex34}) Energy evolutions for BDF3 and BDF4 schemes.} \label{stability}
\end{figure}

We first examine the energy evolution of BDF3A, BDF3B, BDF4A, and BDF4B with $\Delta t=10^{-3}$ and $\Delta t=10^{-4}$, respectively. The numerical results are shown in Fig. \ref{stability}. We find that  BDF4A is unstable, and BDF3A shows oscillations in energy with $\Delta t=10^{-3}$. 
Hence, in the following parts, we will focus on  BDF3B and BDF4B, which, in what follows, are denoted in abbreviation by BDF3 and BDF4.

Then, we examine the numerical errors of BDF3 and BDF4, plotted in Fig. \ref{error34}. 
The reference solution is obtained by ETDRK4 with a sufficiently small time step. 
It is observed that BDF3 and BDF4 schemes achieve the third-order and fourth-order convergence rates, respectively. 

Next, we compare the numerical results of BDF2, BDF3 and BDF4. 
%The reference solution is still obtained by ETDRK4 with a very small time step. 
The energy evolution and the configuration at $t=0.016$ are shown in FIG. \ref{bdf234} (for the first row $\Delta t=10^{-3}$, and for the second row $\Delta t=10^{-4}$). 
We observe that at $\Delta t=10^{-4}$, all schemes lead to the correct solution although there is some visible difference in the energy evolution between BDF2 and the other higher-order schemes, but at $\Delta t=10^{-3}$, only BDF4 leads to the correct solution. 
The above results indicate that higher-order SAV schemes can be used to improve accuracy.

\begin{figure}%[!h]
  \begin{center}
    \includegraphics[width=1.05\textwidth]{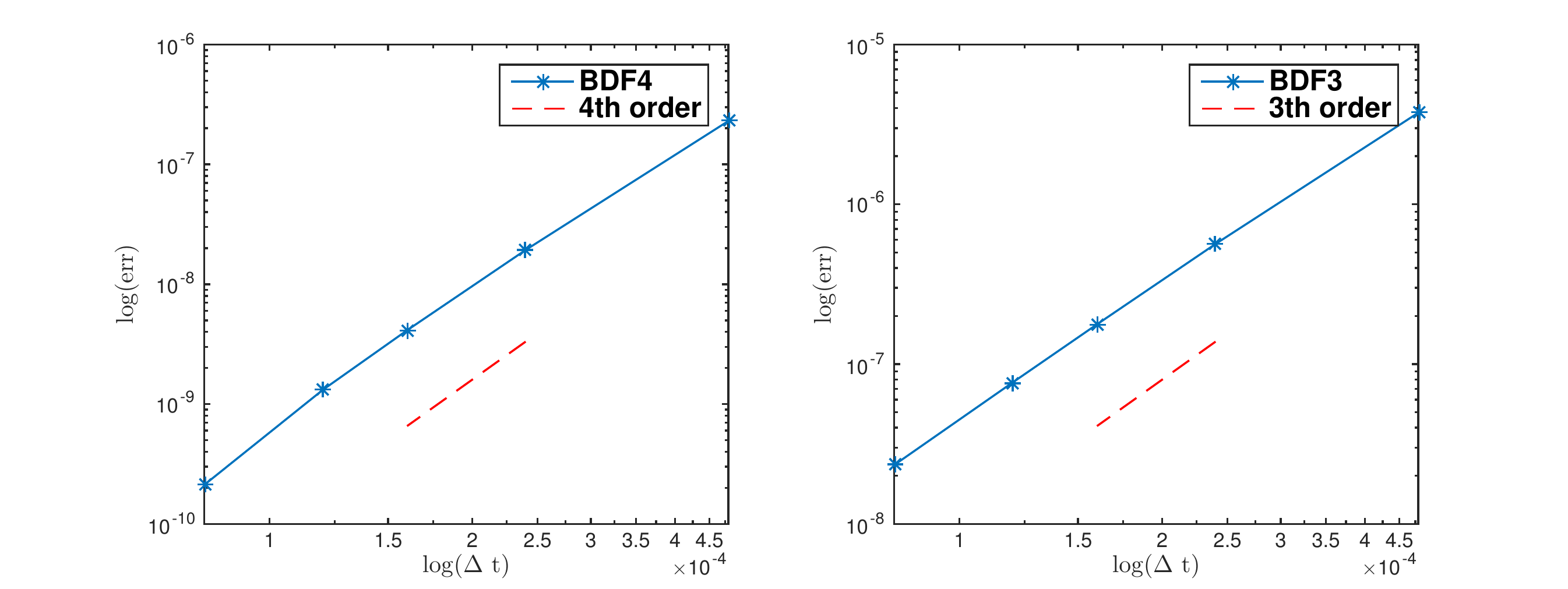}
  \end{center} \vspace{-1em}
  \caption{(Example \ref{ex34}) Numerical convergences of BDF3 and BDF4.} \label{error34}
\end{figure}
\begin{figure}%[!h]
  \begin{center}
    \includegraphics[width=1.05\textwidth]{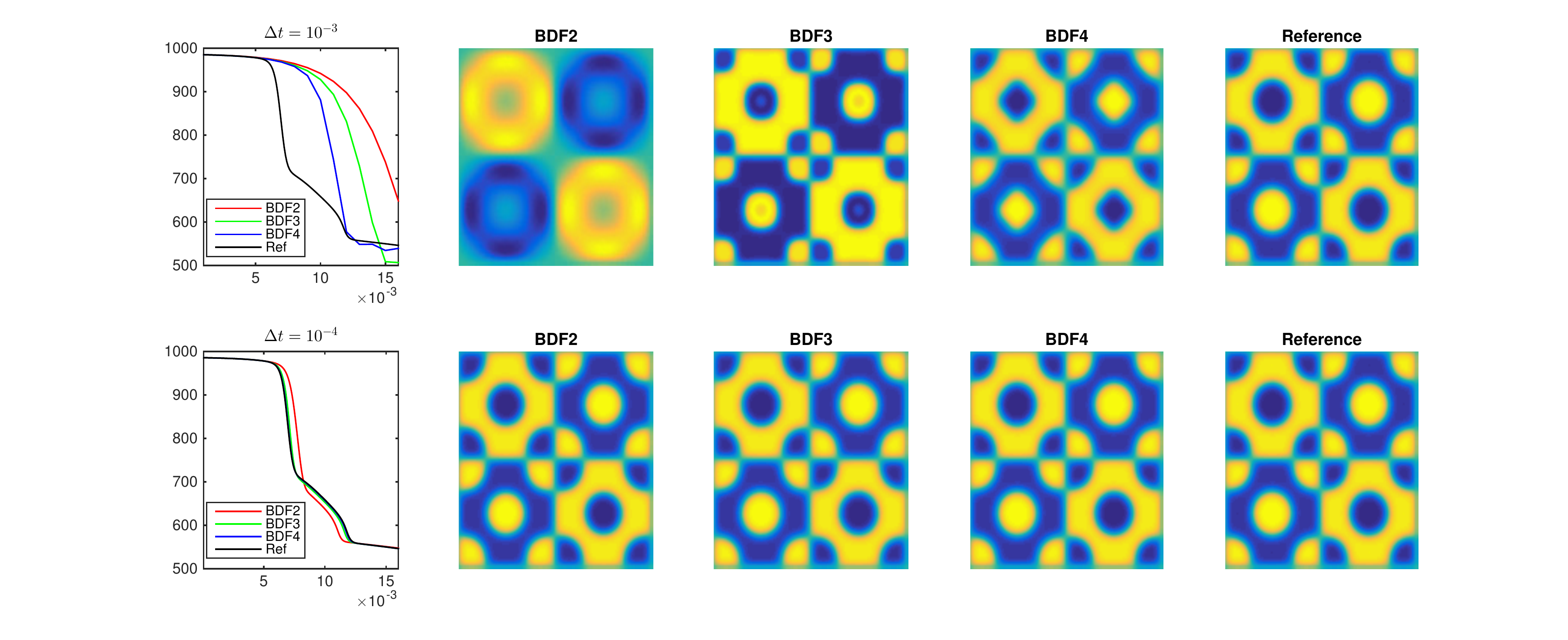}
  \end{center} \vspace{-1em}
  \caption{(Example \ref{ex34}) Comparison of BDF2, BDF3 and BDF4.} \label{bdf234}
\end{figure}

\subsection{Adaptive time stepping}
In many situations, the energy and solution of gradient flows can vary drastically  in a certain time interval, but only slightly elsewhere. 
A main advantage of unconditional energy stable schemes is that they can be easily implemented with an adaptive time stepping strategy so that the time step is only dictated by accuracy rather than by stability as with conditionally stable schemes. 

%This example aims to emphasize the significant impact 	of the energy stable scheme in the practical simulations, especially       being equipped with some time-stepping adaptive strategy.
%	The first order SAV scheme and second order CN/SAV scheme are used in our adaptive time stepping strategy. 	

There are several adaptive strategies for the gradient flows.
Here, we follow the adaptive time-stepping strategy in \cite{shentangyang}, which has been shown to be effective  for Allen--Cahn equations.  We
update the time step size by using the formula
\begin{equation}\label{control}
  A_{\!d\!p}(e,\tau)=\rho \left(\frac{tol}{e}\right)^{1/2}\tau,
\end{equation}
where $\rho$ is a default safety coefficient, $tol$ is a reference tolerance, and $e$ is the relative error at each time level. In this example, we choose 
$\rho=0.9$ and $tol=10^{-3}$. The minimum and
maximum time steps are taken as $\tau_{min}=10^{-5}$ and
$\tau_{max}= 10^{-2}$, respectively.
The initial time step is taken as $\tau_{min}$.
	
\begin{figure}[!ht]
  \begin{center}
    \includegraphics[width=1.05\textwidth]{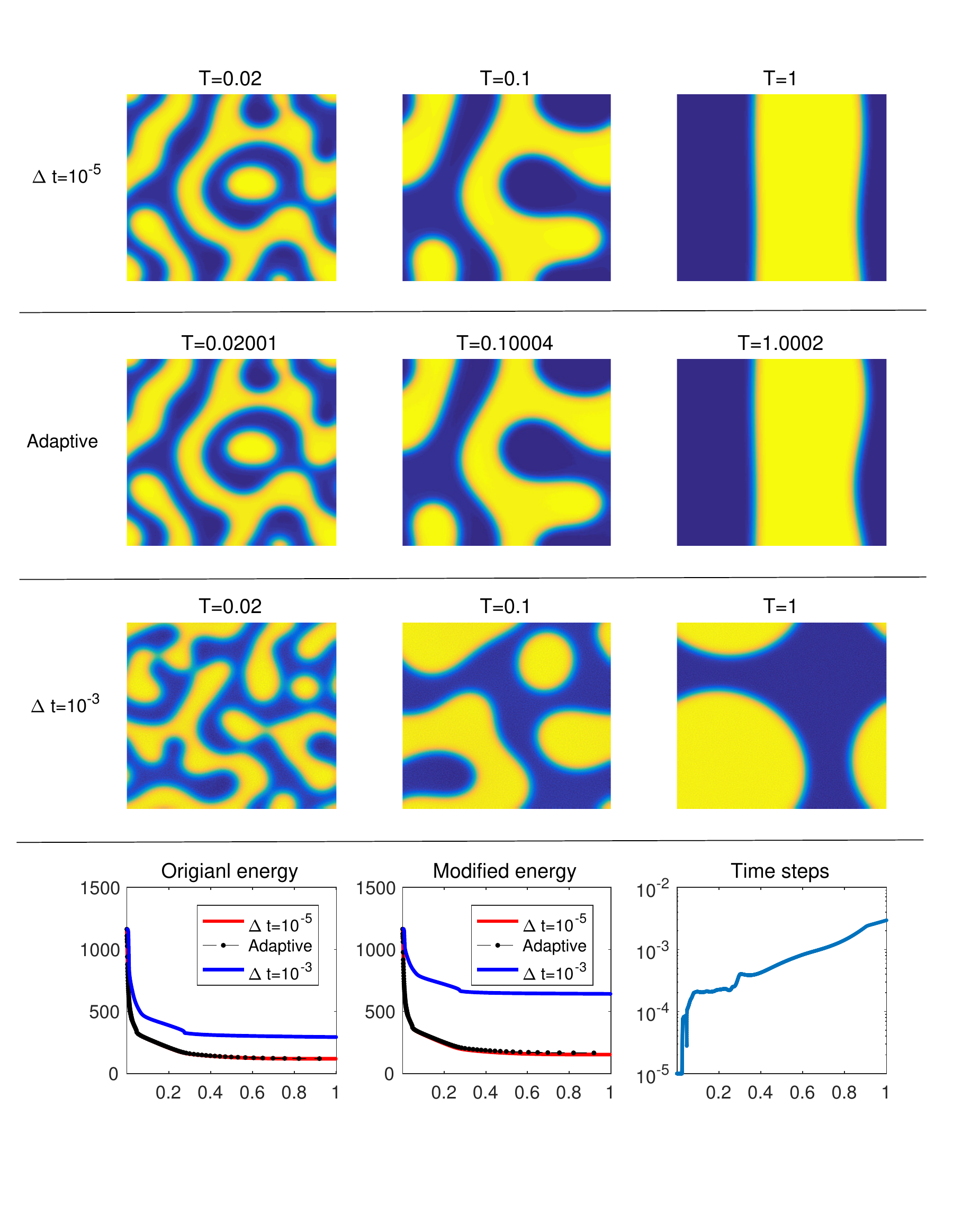}
  \end{center}\vspace{-4em}
  \caption{Example \ref{exadp}:
    Numerical comparisons among small time steps, adaptive time steps, and large time steps}\label{adp}
\end{figure}

{\small
  \begin{algorithm}\noindent
    \caption{Time step and stabilized coefficient adaptive procedure}
    \textbf{Given:} $U^{n}$, $\tau_n$ and stabilized parameter $S_n$.~~~~~~~~~~~~~~~~~~~~~~~~~~~~~~~~~~~~~~~~
    \begin{algorithmic}
      %\WHILE{tol $>$ tolerance,}
      \STATE \textbf{Step 1.} Compute $U^{n+1}_{1}$ by the first order SAV scheme with $\tau_n$.
      \STATE \textbf{Step 2.} Compute $U^{n+1}_2$ by the second order SAV scheme with $\tau_n$.
      \STATE \textbf{Step 3.} Calculate $e_{n+1}=\frac{||U^{n+1}_{1}-U^{n+1}_2||}{||U^{n+1}_2||}$
      \STATE \textbf{Step 4.} $\textbf{if}$ $e_{n+1}>tol$, $\textbf{then}$\\
      \quad\quad\quad\quad\quad Recalculate time step $\tau_n\leftarrow \max\{\tau_{min},\min\{A_{\!d\!p}(e_{n+1},\tau_n), \tau_{max}\}\}$.
      \STATE \textbf{Step 5.} $\textbf{goto}$ Step 1
      \STATE \textbf{Step 6.} $\textbf{else}$\\
      \quad\quad\quad\quad\quad Update time step $\tau_{n+1}\leftarrow \max\{\tau_{min},\min\{A_{\!d\!p}(e_{n+1},\tau_n), \tau_{max}\}\}$.
      \STATE \textbf{Step 7.} $\textbf{endif}$
    \end{algorithmic}
  \end{algorithm}
}
	
We take the $2$D Cahn--Hilliard equation as an example
to demonstrate the  performence of the time adaptivity.
\begin{example}
  \label{exadp}
  Consider the $2$D Cahn--Hilliard equation
  on $[0,2\pi]\times [0,2\pi]$ with periodic boundary conditions and 
  random initial data.
  We take  $\epsilon=0.1$, and  use the Fourier spectral method with $N_x=N_y=256$.
\end{example}

For comparison, we compute a reference solution by the
SAV/CN scheme a small uniform time step $\tau=10^{-5}$ and a large uniform time step $\tau=10^{-3}$. 
Snapshots of phase evolutions, original energy evolutions and modified energy evolution, and the size of time steps in the adaptive experiment are shown in Fig. \ref{adp}. 
It is observed that the adaptive-time solutions given
in the middle row  are in good agreement with the 
reference solution presented in the top row. However, the solutions with large time step are far way from the
reference solution. This is also indicated by both the original energy evolutions and modified energy evolutions.
Note also that
the time step  changes  accordingly with the energy evolution. There are almost three-orders of magnitude variation in the  time step, which indicates 
that  the adaptive time stepping for  the
SAV schemes is very effective. 

\section{Various applications of the SAV approach}
We emphasize that the SAV approach can be applied to a large class of gradient flows. In this section, we shall apply the SAV approach to several challenging gradient flows with different characteristics and show that  the SAV approach leads to very efficient and accurate energy stable  numerical schemes for these problems and those with similar characteristics.

\begin{comment}
\subsection{A gradient flow with both local and non-local dissipation mechanisms}
Nonlocal models, such as those involving fractional Laplacian  or  peridynamic operators, have received much attention recently because their potential in modeling phenomena which can not be accurately described by local models. If we take the dissipation operator $\opL$ in \eqref{Gflow} to be a non-local dissipation operator, then the SAV approach also leads to efficient numerical schemes for non-local gradient flows.

We consider below
a more complicated situation with a mixed dissipation operator $\opL=\opL_l+\opL_\delta$ where $\opL_l$ is a local dissipation operator, and $\opL_\delta$ is a 
non-local 
peridynamic operator \cite{silling2000}:  
$$\opL_\delta \phi=-\int_\Omega \rho_\delta(|\bm y-\bm x|)\left(\phi(\bm y)-\phi(\bm x)\right) d\bm y,$$
where $\rho_\delta$ is a suitable kernal function with $\delta$ being a parameter.
 $\opL_\delta$ can be viewed as the variational derivative of following energy
$$E_\delta(\phi)=\int_\Omega\int_\Omega \rho_\delta(|\bm y-\bm x|)\left(\phi(\bm y)-\phi(\bm x)\right)^2d\bm x d\bm y.$$
\end{comment}

\subsection{Gradient flows with nonlocal free energy}
In most gradient flows, the governing free energy is local, i.e. can be written as an integral of functions about order parameters and their derivatives on a domain $\Omega$. 
Actually, many of these models can be derived as approximations of density functional theory (DFT) (see for example \cite{lutsko2010recent}) that takes a non-local form. 
Recently, there have been growing interests in nonlocal models, aiming to describe phenomena that is difficult to be captured in local models. Examples include peridynamics \cite{silling2000} and quasicrystals \cite{archer2013quasicrystalline,barkan2014controlled,jiang2017stability}. 

Although more complicated forms are possible, we consider the following free energy functional that covers those in the models mentioned above, 
\begin{align}
  \En[\phi]=&\int_{\Omega}\Big(F(\phi)+\frac{1}{2}\phi\opL\phi\Big)\md \bm{x}+\frac{1}{2}\int_{\Omega}\int_{\Omega}K(|\bm{x}-\bm{x'}|)\phi(\bm{x})\phi(\bm{x'})\md\bm{x'}\md\bm{x} \nonumber\\
  :=&(F(\phi),1)+\frac 12(\opL \phi, \phi)+\frac{1}{2}(\phi,\opL_n\phi). 
\end{align}
where $\opL$ is a local symmetric positive differential operator, $K(|\bm{x}-\bm{x'}|)$ is a kernel function, $F(\phi)$ is a nonlinear (local) free energy density, and the operator $\opL_n$ is given by 
\begin{equation}
  (\opL_n\phi)(\bm{x})=\int K(|\bm{x}-\bm{x'}|)\phi(\bm{x'})\md\bm{x'}. 
\end{equation}
Then, the corresponding gradient flow associated with energy dissipation $\opG$ is 
\begin{equation}\label{mixing}
\phi_t=\opG\left(\opL \phi+\opL_n\phi +f(\phi)\right),
\end{equation}
where $f(\phi)=F'(\phi)$.

\begin{comment}
Assume that $c_1$ and $c_2$ are such that  the linear local part will dominate the linear nonlocal part, i.e.,
$c_1(\opL \phi,\phi) > c_2(\opL_\delta \phi,\phi)$.
We can then handle the nonlocal part explicitly in the SAV approach. More precisely, we set 
$$\En_l(\phi)=\frac{c_1}{2}(\opL \phi, \phi) ,\quad \En_n(\phi)=\frac{c_2}{2}E_\delta(\phi)+(F(\phi),1).$$
Assuming$\En_n(\phi)\ge C_0>0$,  we introduce a scalar auxiliary variable
$$r(t)=\sqrt{\En_n(\phi)},$$
and rewrite  the gradient flow \eqref{mixing} as
\begin{subequations}\label{mixSAV}
\begin{align}
\phi_t&=\opG\left(c_1\opL \phi+\frac{r}{\sqrt{\En_n(\phi)}}\left(c_2\opL_\delta \phi +f(\phi)\right)\right),\\
r_t&=\frac{1}{2\sqrt{\En_n(\phi)}}\left(\phi_t,c_2\opL_\delta \phi +f(\phi)\right).
\end{align}
\end{subequations}
Now we can construct the scheme based on SAV/BDF1
\begin{align*}
  \frac{\phi^{n+1}-\phi^n}{\Delta t}=&\opG\mu^{n+1}, \\
  \mu^{n+1}=&\opL \phi^{n+1}
  +\frac{r^{n+1}}{\sqrt{\En_n[\phi^{n}]}}\left(c_2\opL_\delta \phi^{n} +f(\phi^{n})\right), \\
r^{n+1}-r^n=&\frac{1}{2\sqrt{\En_n[\phi^{n}]}}\Big(c_2\opL_\delta \phi^{n} +f(\phi^{n}),\phi^{n+1}-\phi^n\Big). 
\end{align*}
\end{comment}

In general, $\opL$ may not be positive and shall be controlled by the nonlinear term $F(\phi)$, as  in the non-local models we mentioned above. 
In this case, we may put part of the non-local term together with the nonlinear term, 
and handle the non-local term explicitly in the SAV approach. 
More precisely, we split $\opL_n=\opL_{n1}+\opL_{n2}$ set 
$$
\En_l(\phi)=\frac{1}{2}(\opL \phi, \phi)+\frac 12(\phi,\opL_{n1}\phi) ,\quad \En_n(\phi)=\frac{1}{2}(\phi,\opL_{n2}\phi)+(F(\phi),1), 
$$
where we assume that $\opL_{n1}$ is positive and $\En_n(\phi)\ge C_0>0$. 
We introduce a scalar auxiliary variable 
$$
r(t)=\sqrt{\En_n(\phi)},
$$
and rewrite  the gradient flow \eqref{mixing} as
\begin{subequations}\label{mixSAV}
\begin{align}
\phi_t&=\opG\left((\opL+\opL_{n1})\phi+\frac{r}{\sqrt{\En_n(\phi)}}\left(\opL_{n2} \phi +f(\phi)\right)\right),\\
r_t&=\frac{1}{2\sqrt{\En_n(\phi)}}\left(\phi_t,\opL_{n2} \phi +f(\phi)\right).
\end{align}
\end{subequations}
Then the second-order BDF scheme based on SAV approach is:
\begin{subequations}\label{mixSAV2}
\begin{align}
&  \frac{3\phi^{n+1}-4\phi^n+\phi^{n-1}}{2\Delta t}=\opG\mu^{n+1}, \\
&  \mu^{n+1}=(\opL+\opL_{n1}) \phi^{n+1}
  +\frac{r^{n+1}}{\sqrt{\En_n[\bar{\phi}^{n+1}]}}\left(\opL_{n2} \bar{\phi}^{n+1} +f(\bar{\phi}^{n+1})\right), \\
&3r^{n+1}-4r^n+r^{n-1}=\frac{1}{2\sqrt{\En_n[\bar{\phi}^{n+1}]}}\Big(\opL_{n2} \bar{\phi}^{n+1} +f(\bar{\phi}^{n+1}),3\phi^{n+1}-4\phi^n+\phi^{n-1}\Big). 
\end{align}
\end{subequations}
Similarly, it is easy to show that the above scheme is unconditionally energy stable, and that the scheme only requires, at each time step,  solving two linear systems of the form:
\begin{equation}\label{common2}
(I-\lambda\Delta t\opG(\opL+\opL_{n1}))\phi=f.
\end{equation}
In particular, if $\opL>\opL_n$, a good choice can be $\opL_{n1}=0$ and $\opL_{n2}=\opL_n$, and only need to solve equations with common differential operators. Note also that the phase field crystal model considered in Subsection 3.2 is a special case with $\opL=0$ and $\opL_{n2}=0$.

Note that the above problem can not be easily treated with convex splitting or IEQ approaches.
%In many situations, \eqref{common2} can be efficiently solved

%\textit{Remark.} In the above, we treated the double integral in $\En_n$ explicitly,??? which may not be written as an integral of a positive function. 
%This would result in failure in using IEQ approach, but the SAV approach can  be applied easily and efficiently, even when the double integral contains nonlinear terms or $\En_n$ has triple or quadruple integrals. 

\subsection{Molecular beam epitaxial (MBE) without slope selection}
The energy functional for molecular beam epitaxial (MBE) without slope selection is given by \cite{Li.L03}:
\begin{equation}
\En(\phi)=\int_{\Omega}[-\frac{1}{2}\ln(1+|\nabla \phi|^2)+\frac{\eta^2}{2}|\Delta \phi|^2]dx.
\end{equation}
A main difficulty is that the first part of the energy density, $-\frac{1}{2}\ln(1+|\nabla \phi|^2)$, is unbounded from below, so the IEQ approach can not be applied. However, one can show that \cite{CSY17} for any $\alpha_0>0$, there exist $C_0>0$ such that
\begin{equation}\label{ori:ene1}
\En_1(\phi)=\int_{\Omega}[-\frac{1}{2}\ln(1+|\Delta \phi|^2)+\frac{\alpha}{2}|\Delta \phi|^2]dx \ge -C_0,\quad \forall \alpha\ge \alpha_0>0.
\end{equation}
Hence, we can choose $\alpha_0<\alpha<\eta^2$, and split $\En(\phi)$ as
$$\En(\phi)=\En_1(\phi)+\int_{\Omega} \frac{\eta^2-\alpha}{2}|\Delta \phi|^2dx.$$
Now we introduce a scalar auxiliary variable
$$r(t)=\sqrt{\int_{\Omega}\frac{\alpha}{2}|\Delta\phi|^2-\frac{1}{2}\log(1+|\nabla\phi|^2)dx+C_0},$$
and rewrite the gradient flow for MBE as
\begin{subequations}\label{new2}
\begin{align}
&\phi_t+M[(\eta^2-\alpha)\Delta^2\phi+ G(\phi){r(t)}]=0, \label{new2:no1}\\
&r_t=\frac{1}{2}\int_{\Omega} G(\phi) \phi_t dx, \label{new2:no3}
\end{align}
\end{subequations}
where
$$G(\phi)=\frac{\frac{\delta E_1(\phi)}{\delta \phi}}{\sqrt{\int_{\Omega}\frac{\alpha}{2}|\nabla\phi|^2-\frac{1}{2}\log(1+|\nabla\phi|^2)dx+C_0}}.$$
Therefore, we can use the SAV approach to construct, for \eqref{new2}, second-order, linear, unconditionally energy stable schemes  which only require, at each time step, solving two linear equations of the form 
$$(I+\Delta t\Delta^2)\phi=f.$$
It is clear that this above approach is much more efficient and easier to implement than existing schemes (cf., for instance, \cite{Li.L03,shen2012second,MR3290959}). 

\subsection{Q-tensor model for rod-like liquid crystals}
In many liquid crystal models, a symmetric traceless second-order tensor $Q\in \mathbb{R}^{3\times 3}$ is used to described the orientational order. 
We consider the Landau-de Gennes free energy \cite{de1971short} that has been applied to study various phenomena, both analytically (see for example \cite{majumdar2012radial,paicu2012energy}) and numerically (see for example \cite{sheng1982boundary,schopohl1987defect,zhao2017novel}). 
It can be written as $\En[Q(\bm{x})]=\En_{b}+\En_{e}$, where
\begin{align}
  \En_b&=\int_{\Omega}f_b(Q)\md\bm{x}=\int_\Omega[\frac{a}{2}\mbox{tr}(Q^2)-\frac{b}{3}\mbox{tr}(Q^3)+\frac{c}{4}(\mbox{tr}(Q^2))^2]\md\bm{x}, \\
  \En_e&=\int_\Omega[\frac{L_1}{2}|\nabla Q|^2+\frac{L_2}{2}\partial_{i}Q_{ik}\partial_{j}Q_{jk}
  +\frac{L_3}{2}\partial_{i}Q_{jk}\partial_{j}Q_{ik}]\md\bm{x}. 
\end{align}
To ensure the lower-boundedness, it requires $c>0,\;L_1,L_1+L_2+L_3>0$ so that we have  $\En_b,\,\En_e\ge 0$. 

We consider the $L^2$ gradient flow, 
\begin{align}\label{GflowQ}
  \frac{\partial Q_{ij}}{\partial t}=-\frac{\delta \En}{\delta Q_{ij}},
  \quad 1\le i,j\le 3,
\end{align}
with 
\begin{align}
  \frac{\delta \En_b}{\delta Q_{ij}}=&aQ_{ij}-bQ_{ik}Q_{kj}+(c+\frac{b}{3})Q_{ij}\mbox{tr}(Q)^2, \\
  \frac{\delta \En_e}{\delta Q_{ij}}=&-L_1\Delta Q_{ij} - 
  \frac{L_2+L_3}{2}(\partial_{ik}Q_{jk}+\partial_{jk}Q_{ik}-\frac{2}{3}(\partial_{kl}Q_{kl})\delta_{ij}). 
\end{align}
%?? Add some mathematical and numerical references, explain the difficulties!??
We can see that the components of $Q$ are coupled both in $\En_b$ and $\En_e$, which makes it difficult to deal with  numerically. 

Since we have a positive quartic term $c(\mbox{tr}Q^2)^2$, we can choose $a_1,C_0>0$ such that $f_b(Q)-a_1\mbox{tr}(Q^2)/2+C_0>0$. 
We introduce a scalar auxiliary variable
$$
r(t)=\sqrt{\En_1}:=\sqrt{\En_b(Q)-\int_\Omega\frac{a_1}{2}\mbox{tr}(Q^2) \md\bm{x}+C_0}. 
$$
Let $\opL$ be defined as 
$$
\opL Q=a_1Q+\frac{\delta\En_e}{\delta Q}, 
$$
where $\delta\En_e/\delta Q:=(\delta\En_e/{\delta Q_{ij}})$ defines a linear operator on $Q$. 
%$\opG=-1$ since we are considering the $L^2$ gradient flow, 
Hence, we can rewrite \eqref{GflowQ} as:
\begin{equation}
 \label{GflowQb}
\begin{split}
  &\frac{\partial Q}{\partial t}=-\mu,\\
 &\mu= \opL Q +\frac{r(t)}{\sqrt{\En_1}}\frac{\delta\En_1}{\delta Q};\\
 &r_t=\frac1{2\sqrt{\En_1}}\int_\Omega \frac{\delta\En_1}{\delta Q}:Q_t \md\bm{x}.
\end{split}
\end{equation}
 Then, the SAV/CN scheme for \eqref{GflowQb} is:
\begin{subequations} \label{GflowQ2}
\begin{align}
  \frac{Q^{n+1}-Q^n}{\Delta t}=&-\mu^{n+1/2}, \\
  \mu^{n+1/2}=&\opL\frac{1}{2}(Q^{n+1}+Q^n)
  +\frac{r^{n+1}+r^n}{2\sqrt{\En_1[\bar{Q}^{n+1/2}]}}
  \frac{\delta\En_1}{\delta Q}[\bar{Q}^{n+1/2}], \\
  r^{n+1}-r^n=&\int_\Omega\frac{1}{2\sqrt{\En_1[\bar{Q}^{n+1/2}]}}\left(\frac{\delta\En_1}{\delta Q}[\bar{Q}^{n+1/2}]\right)_{ij}
  (Q_{ij}^{n+1}-Q_{ij}^n)\md\bm{x}. 
\end{align}
\end{subequations}
One can easily show that the above scheme is unconditionally energy stable.
 Below, we describe how to implement it efficiently. 
 
 Denoting
$$
S_{ij}=\frac{1}{2\sqrt{\En_1[\bar{Q}^{n+1/2}]}}\left(\frac{\delta\En_1}{\delta Q}[\bar{Q}^{n+1/2}]\right)_{ij},\quad S=(S_{ij}),
$$
 we can rewrite \eqref{GflowQ2} into  a coupled linear system of the form 
\begin{align}\label{qij}
  (1+\lambda \opL)Q_{ij}^{n+1}+S_{ij}\alpha^{n+1}=b_{ij}^n, \quad 1\le i,j\le 3, 
\end{align}
where $\lambda=\frac{\Delta t}2$, and the scalar $\alpha^{n+1}=\sum_{k,l}(S_{kl},Q_{kl}^{n+1})$ can be solved explicitly as follows. 
Multiplying \eqref{qij} with $(1+\lambda\opL)^{-1}$, we get
\begin{align}
  Q_{ij}^{n+1}+\alpha^{n+1}\big((I+\lambda \opL)^{-1}S\big)_{ij}=\Big((1+\lambda \opL)^{-1}b^n\Big)_{ij}, \quad 1\le i,j\le 3.
\end{align}
Then taking the inner product of the above with $S$, we obtain 
\begin{align}
  \alpha^{n+1}\bigg(1+\sum_{i,j}\Big(S_{ij},\big((1+\lambda\opL)^{-1}S\big)_{ij}\Big)\bigg)=\sum_{i,j}\Big((I+\lambda\opL)^{-1}b^n\Big)_{ij}. 
\end{align}
Thus, we can find $\alpha^{n+1}$ by solving, for each $i,j$, two equations of the form 
\begin{equation}\label{Qij2}
 (I+\lambda\opL)Q_{ij}=g_{ij},
\end{equation}
which can be efficiently solved since they are simply coupled second-order equations with constant coefficients. For example, 
in the case of periodic boundary conditions, we can write down the solution explicitly as follows. 
Because $Q$ is symmetric and traceless, we choose $\bm{x}=(Q_{11},Q_{22},Q_{12},Q_{13},Q_{23})^T$ as independent variables. 
We expand the above five variables by Fourier series, 
$$
Q_{ij}=\sum_{k_1,k_2,k_3}\hat{Q}_{ij}^{k_1k_2k_3}\exp(i(k_1x_1+k_2x_2+k_3x_3)). 
$$
Then, when solving the linear equation \eqref{Qij2}, only the Fourier coefficients with the same indices $(k_1,k_2,k_3)$ are coupled. More precisely, for each  $(k_1,k_2,k_3)$, and the coefficient matrix for  the unknowns  $\hat{Q}_{ij}^{k_1k_2k_3}$ with $(ij=11,22,12,13,23)$ is given by 
\begin{align*}
&A_{k_1k_2k_3}=1+\lambda(a_1+L_1(k_1^2+k_2^2+k_3^2))I\nonumber\\
&-\lambda\frac{L_2+L_3}{2}
\left(
\begin{array}{ccccc}
  -\frac{2}{3}k_1^2-\frac{1}{3}k_3^2 & \frac{1}{3}k_2^2-\frac{1}{3}k_3^2 & -\frac{1}{3}k_1k_2 & -\frac{1}{3}k_1k_3 & \frac{2}{3}k_2k_3\\
  \frac{1}{3}k_1^2-\frac{1}{3}k_3^2 & -\frac{2}{3}k_2^2-\frac{1}{3}k_3^2 & -\frac{1}{3}k_1k_2 & \frac{2}{3}k_1k_3 & -\frac{1}{3}k_2k_3\\
  -\frac{1}{2}k_1k_2 & -\frac{1}{2}k_1k_2 & -\frac{1}{2}k_1^2-\frac{1}{2}k_2^2 & -\frac{1}{2}k_2k_3 & -\frac{1}{2}k_1k_3 \\
  0 & \frac{1}{2}k_1k_3 & -\frac{1}{2}k_2k_3 & -\frac{1}{2}k_1^2-\frac{1}{2}k_3^2 & -\frac{1}{2}k_1k_2 \\
  \frac{1}{2}k_2k_3 & 0 & -\frac{1}{2}k_1k_3 & -\frac{1}{2}k_1k_2 & -\frac{1}{2}k_2^2-\frac{1}{2}k_3^2 
\end{array}
\right). 
\end{align*}
Hence, we can obtain the Fourier coefficients $\hat{Q}_{ij}^{k_1k_2k_3}$, for each $i,\,j$, by inverting the above $5\times 5$ matrix.
%Apparently, the linear equation is asymmetric. But in our scheme it can still be solved easily. 

\subsection{Phase-field model of two phase incompressible flows}
We consider here a phase-field model for the mixture of two incompressible, immiscible  fluids \cite{shen2015decoupled}. Let $\phi$ be a labeling function to identify the two fluids, i.e.,
\begin{equation}
 \phi({x},t)=\begin{cases} \hskip 8pt 1\quad {x}\in \text{in fluid 1},\\
 -1\quad {x}\in \text{in fluid 2},\end{cases}
\end{equation}
with a smooth interfacial layer of thickness $\eta$. 
 Consider a mixing free energy
\begin{equation*}
E_{mix}(\phi) =\lambda
\int_\Omega (\frac
12|\nabla \phi|^2+F(\phi))\, \md \bm{x}=\lambda
\int_\Omega \frac
12|\nabla \phi|^2\, \md \bm{x} + E_1(\phi),
\end{equation*}
with $F(\phi)=\frac1{4\eta^2}(\phi^2-1)^2$. For the sake of simplicity, we consider the two fluids having the same density $\rho_0$. Then, the Navier-Stokes Cahn-Hilliard phase field model for the two-phase incompressible flow is as follows (cf., for instance, \cite{anderson1998diffuse,liu2003phase}):
\begin{eqnarray} \label{phase1}
%\begin{split}
&\phi_t + (u\cdot\nabla) \phi = \nabla \cdot(\gamma \nabla \mu),\\
&\mu=\frac{\delta
  E_{mix}}{\delta\phi}=-\lambda\Delta \phi
+\lambda F'(\phi); \label{phase2}
%\end{split}
\end{eqnarray}
and
\begin{equation} \label{vel}
\rho_0(u_t + (u\cdot\nabla)u)=\nu\Delta
u-\nabla p+\mu\nabla \phi;
\end{equation}
and 
\begin{equation}\label{divfree}
 \nabla \cdot u = 0;
\end{equation}
subject to suitable boundary conditions for $\phi,\mu,u$. In the above, $\lambda$ is a mixing coefficient, $\gamma$ is a relaxation coefficients and $\nu$ is the viscosity coefficient; the unknown are $\phi,\mu,u,p$ with $u$ being the velocity and $p$ the pressure.
Taking the inner product of \eqref{phase1},  \eqref{phase2} and  \eqref{vel} with $\mu$, $\phi_t$ and $u$, respectively, we  obtain the following energy dissipation law:
\begin{equation*}\label{energy}
\frac{\md}{\md t}\int_\Omega \{
\frac{\rho_0}{2}\md \bm{x}\,|u|^2+\frac{\lambda}2|\nabla \phi|^2+ 
 {\lambda}F(\phi) \} 
= -\int_\Omega \{\nu|\nabla u|^2+\gamma |\nabla \mu|^2\}.
\end{equation*}
To apply the SAV approach, we introduce $r(t)=\sqrt{E_1(\phi)+\delta}$, and replace  \eqref{phase2} by 
\begin{equation}\label{phase2b}
 \begin{split}
& \mu=-\lambda\Delta \phi
+\lambda\frac{r(t)}{\sqrt{E_1(\phi)+\delta}} F'(\phi),\\
& r_t=\frac1{2\sqrt{E_1(\phi)+\delta}}\int_\Omega F'(\phi)\,\phi_t\md\bm{x}.
\end{split}
\end{equation}

Let us denote $\bar\phi^{n+1}:=2\phi^n-\phi^{n-1}$ and $\bar u^{n+1}:=2u^n-u^{n-1}$. 
Then, we can modify the scheme (3.9) in \cite{She.Y10b} to construct the SAV/BDF2 scheme for \eqref{phase1}-\eqref{phase2b}-\eqref{vel}-\eqref{divfree}:

\begin{equation}\label{phasefiled1}
\begin{split}
&\frac{3\phi^{n+1}-4\phi^n+\phi^{n-1}}{2\Delta t}+
  \hat u^{n+1}\cdot \nabla\bar\phi^{n+1}=\gamma\Delta
  \mu^{n+1},\\
   & \mu^{n+1}=-\lambda\Delta\phi^{n+1}+\frac{\lambda r^{n+1}}{\sqrt{E_1[\bar\phi^{n+1}]+\delta}}F'(\bar\phi^{n+1}),\\
&  \frac{3r^{n+1}-4r^n-r^{n-1}}{2\Delta t}=
  \int_\Omega \frac{F'(\bar\phi^{n+1})}{2\sqrt{E_1[\bar\phi^{n+1}]+\delta}}\frac{3\phi^{n+1}-4\phi^{n}+\phi^{n-1}}{2\Delta t} \, dx, 
  \end{split}
\end{equation}
where $\hat u^{n+1}=\tilde u^{n+1}$ or $\hat u^{n+1}=2u^n-u^{n-1}$;
\begin{equation}
\begin{split}
&\rho_0\{\frac{3\tilde u^{n+1}-4u^n+u^{n-1}}{2\Delta t}+\bar u^{n+1}\cdot\nabla
  \tilde u^{n+1}\}\\ &\hskip .5in
  -\nu\Delta \tilde u^{n+1} +\nabla p^n-\mu^{n+1}\nabla \bar\phi^{n+1}=0;
\end{split}
\end{equation}

\begin{equation}
\begin{split}
&\Delta (p^{n+1}-p^n)=\frac{3\rho_0}{2\Delta t}\nabla\cdot \tilde u^{n+1}, \quad
\partial_n (p^{n+1}-p^n)|_{\partial\Omega}=0;\\
&u^{n+1}=\tilde u^{n+1}-\frac{2\Delta t}{3\rho_0} \nabla (p^{n+1}-p^n).
\end{split}
\end{equation}
Several remarks are in order:
\begin{itemize}
 
\item The pressure is decoupled from the rest by a pressure-correction projection method;  $r^{n+1}$  can be eliminated from \eqref{phasefiled1}.

 \item If we take $\hat u^{n+1}=\tilde u^{n+1}$, one can show, similarly as in \cite{She.Y10b},  that the scheme is unconditionally stable,  linear and second-order, but weakly coupled between $(\phi^{n+1},w^{n+1},\tilde u^{n+1})$ by the term $\tilde u^{n+1}\cdot \nabla\bar\phi^{n+1}$. It is easy to see that the weakly coupled linear system is positive definite.
 
 \item On the other hand, if we take $\hat u^{n+1}=2u^n-u^{n-1}$, the scheme is   linear, decoupled and second-order, only requires solving a sequence of Poisson type equations at each time step, but not unconditionally energy stable.
 
 \item One can use the decoupled  scheme with $\hat u^{n+1}=2u^n-u^{n-1}$ as a preconditioner for the coupled scheme with $\hat u^{n+1}=\tilde u^{n+1}$.
\end{itemize}

\section{Conclusion\label{Concl}}
We proposed a new SAV approach for dealing with a large class of  gradient flows. This approach  keeps all advantages of the IEQ approach, namely, the schemes are unconditionally energy stable, linear and second-order accurate, while  offers the following additional advantages: 
\begin{itemize}
\item It greatly simplifies the implementation and is much more efficient:
at each time step of the SAV schemes, the computation of 
 the scalar auxiliary  variable $r^{n+1}$ and the original unknowns are totally decoupled and only requires solving  linear systems with  constant coefficients.
\item   It only requires  $E_1(\phi):=\int_\Omega F(\phi) \md\bm{x}$, instead of $F(\phi)$, be bounded from below. It also allows the energy functional contains multiple integrals. Thus it applies to  a larger class of gradient flows. 
\item It offers an effective approach to deal with gradient flows with non-local free energy.
\end{itemize}
Furthermore,  we can even construct higher-order stiffly stable, albeit not unconditionally stable,  schemes with all the  above attributes by combining SAV approach with higher-order BDF schemes. 

When coupled with a suitable time adaptive strategy, the SAV schemes are extremely efficient and  applicable to a large class of gradient flows. 
Our numerical results show that the SAV schemes are not only more efficient, but are also more accurate than other schemes.

Although the SAV approach appears to be applicable and very effective for a large class of gradient flows,  an essential requirement for the SAV approach to produce physically consistent results is  that  $\opL$ in the energy splitting  \eqref{energy0}   contains enough  dissipative terms (with at least linearized highest derivative terms).
To be specific, $\En_1[\phi]$ shall not include all terms with the highest derivatives in $\En[\phi]$, i.e., $\opL$ has to include a term, if not all, with the highest derivatives.
 
We have focused in this paper on gradient flows with linear dissipative mechanisms. For problems with highly nonlinear   dissipative mechanisms, e.g., $\opG \mu=\nabla\cdot(a(\phi)\nabla \mu)$ with degenerate or singular $a(\phi)$ such as in Wasserstein   gradient flows or  gradient flows with strong anisotropic free energy \cite{chen2013efficient}, the direct application of SAV approach may not be the most  efficient as it leads to  degenerate or singular  nonlinear equations to solve at each time step. In \cite{She.X17}, we developed an efficient predictor-corrector strategy to deal with this type of problems without the need to solving nonlinear equations.

While it is important that  numerical schemes for gradient flows  obey a discrete energy dissipation law, 
 the energy dissipation itself does not guarantee the convergence. In a future work, we shall conduct an error analysis for the SAV approach  and identify sufficient conditions under which the SAV schemes will converge, with an error estimate, to the exact solution of the original problem.

\bibliographystyle{siamplain}
\bibliography{bib_gflow}

\end{document}